\documentclass{article}
\usepackage{fullpage} 
\usepackage{amsmath}   
\usepackage{amssymb}   
\usepackage{amsthm}    
\usepackage{amsfonts}

\usepackage[latin1]{inputenc}   

\usepackage[capitalise]{cleveref}
\usepackage{bbm}
\usepackage[T1]{fontenc}   
\usepackage{algpseudocode}
\usepackage{tikz}
\usepackage{amsmath}
\usepackage{pgf,tikz}
\usepackage{graphicx} 
\graphicspath{{./}{figures/}} 
\usetikzlibrary{arrows}
\usepackage{amsopn}
\usepackage{adjustbox}

\def\rg{\rangle}
\def\lg{\langle} 
\def \to{\rightarrow}
\def \states{\mathbb{T}^d}
\def \T{\mathbb{T}}
\def \R {\mathbb{R}}
\def \N {\mathbb{N}}

\def \mes{\mathcal{P}}
\def\Pk{\mes(\states)}
\def \mint{\int_{\states}}

\def \d{\mathrm{d}}

\def \abs{|}

\def \leftnorm{\left|}
\def \rightnorm{\right|}

\newcommand{\be}{\begin{equation}}
\newcommand{\ee}{\end{equation}}

\theoremstyle{definition}

\theoremstyle{plain}
\newtheorem{theorem}{Theorem}[section]
\newtheorem{proposition}[theorem]{Proposition}

\newtheorem{definition}[theorem]{Definition}
\newtheorem{lemma}[theorem]{Lemma}

\theoremstyle{remark}
\newtheorem{remark}[theorem]{Remark}

\numberwithin{equation}{section}

\begin{document}
\title{Learning optimal policies in potential Mean Field Games: Smoothed Policy Iteration algorithms}
\date{}
\author{Qing Tang\thanks{China University of Geosciences (Wuhan), China. tangqingthomas@gmail.com}, \,\,
Jiahao Song\thanks{China University of Geosciences (Wuhan), China. songjh@cug.edu.cn}}
\maketitle

\begin{abstract}
We introduce two Smoothed Policy Iteration algorithms (\textbf{SPI}s) as rules for learning policies and methods for computing Nash equilibria in second order potential Mean Field Games (MFGs). Global convergence is proved if the coupling term in the MFG system satisfy the Lasry Lions monotonicity condition. Local convergence to a stable solution is proved for system which may have multiple solutions. The convergence analysis shows close connections between \textbf{SPI}s and the Fictitious Play algorithm, which has been widely studied in the MFG literature. Numerical simulation results based on finite difference schemes are presented to supplement the theoretical analysis.
\end{abstract}

\section{Introduction}

Mean Field Games (MFG for short) theory has been introduced in \cite{hcm,ll} to characterize Nash equilibria for differential games involving a large (infinite) number of symmetric and strategic agents. For a comprehensive introduction to the applications of MFG theory we refer to the monographs by Carmona and Delarue \cite{carmona2018}, Bensoussan, Frehse and Yam~\cite{ben} and the lecture notes~\cite{achdouCetraro,lauriere2021}. 
In this paper, we focus on considering the evolutive second order potential MFG model:
\be\label{MFG}
\left\{\begin{split}
(i) \qquad &-\partial_t u(t,x) - \sigma \Delta u(t,x)  +H(x,\nabla u(t,x)) =f[m(t)](x), &{\rm in}\,\,Q,\\
(ii) \qquad &\partial_t m(t,x)- \sigma \Delta m(t,x)  - {\rm{div}} ( m(t,x)H_p (x,\nabla u(t,x))) =0, &{\rm in}\,\,Q, \\
&m(0,x)=m_0(x) , \; u(T, x)=g[m (T)](x)&{\rm in}\,\,\T^d.
\end{split}\right.
\ee
 The $(i)$ in \eqref{MFG} is the Hamilton-Jacobi-Bellman (HJB for short) equation characterizing the value function $u(t,x)$  for a representative agent at time $t$ and state $x$, solving the stochastic optimal control problem (with $d-$dimensional Brownian motion and $\sigma>0$):  $dX_\tau= -q_\tau\, d\tau + \sqrt{2\sigma}\, dB_\tau,\,\,X_t= x$, 
$$
 u(t,x)= \inf_{q_\tau \in \R^d} {\mathbb E} \left\{ \int_t^T \big(L(X_\tau, q_\tau)+ f[m(\tau)](X_\tau)\big) dt+ g[m(T)](X_T) \right\},
$$
with control process $q_\tau$. $L$, $f$, $g$ stand for costs. Define the Hamiltonian
\begin{equation}\label{defH}
H(x,p)=\sup_{q\in \R^d}p\cdot q-L(x,q),
\end{equation}
 $q^*_t=H_p(X_t,\nabla u(t,X_t))$ is the optimal policy (control in feedback form) at time $t$. $(ii)$ is the Fokker-Planck-Kolmogorov (FPK for short) equation, governing the density of the population $m(t,x)$, with each agent following the policy $q^*$. $(t,x)\in Q$ where $Q:=[0,T]\times \T^d$ and $\T^d$ stands for the flat torus $\R^d / \mathbb{Z}^d$. \par
\textbf{Literature.} The forward backward structure of system \eqref{MFG} precludes a simple time marching algorithm for solving its discretized form. Numerical methods for MFGs with finite difference schemes goes back to the pioneering works of Achdou and I. Capuzzo-Dolcetta \cite{ad,acd}. In \cite{ad,acd}, the discretization is based on on a monotone upwind scheme for the HJB equation, while the FPK is obtained by differentiating the discrete HJB equation and taking the adjoint. This nonlinear system can be solved by Newton algorithm, which has local and quadratic convergence. In \cite{ccg}, Cacace, Camilli and Goffi proposed solving the nonlinear system by policy iteration method. In optimal control problems, policy iteration (also known as Howard algorithm \cite{howard1960}) stands for an iterative method with index $n$, of alternating between evaluating the value $u^{(n)}$ of a policy $q^{(n)}$ and generating a new policy $q^{(n+1)}$, e.g. \cite[pp. 521-523]{bl}. In \cite{ccg}, an update of distribution $m^{(n)}$ is generated by the policy $q^{(n)}$, before $u^{(n)}$ is evaluated. With some quantitative assumptions, one can show that the policy iteration operator in \cite{ccg,ct} and its generalizations \cite{lst} are contractive and thereby obtain a linear rate of convergence.  \par
 
 Fictitious Play algorithm (\textbf{FP} for short) is a widely used model of learning in games for explaining how and why equilibria arises as well as computing Nash equilibria, e.g. \cite{FL98,MondrerShapley2}. \textbf{FP} for  potential MFGs was first introduced in \cite{ch} by Cardaliaguet and Hadikhanloo. Numerical implementation of \textbf{FP} for MFGs were considered in \cite[pp. 8-9]{lauriere2021}. The theory of \textbf{FP} has been further developed in \cite{bc,dumitrescu2022linear,Hadik,hs} for the MFG PDE systems and in \cite{elie2019approximate,l2022} for MFG reinforcement learning problems. Local convergence of \textbf{FP} to a stable solution of a nonconvex potential MFG was considered in \cite{bc} and the selection of equilibria with \textbf{FP} and common noise was considered by Delarue and Vasileiadis in \cite{delarue2021}. In \cite{blp}, Lavigne and Peiffer showed the connections between \textbf{FP} and the conditional gradient method (Frank-Wolf algorithm) for solving MFGs, with an exhaustive discussion on learning rates and convergence rates. Numerical methods were considered in \cite{blp} using explicit finite difference schemes.  \par
\textbf{Contributions.} Our main contributions are introducing Smoothed Policy Iteration algorithms (henceforth referred to as \textbf{SPI}s) for solving system \eqref{MFG}. Our main assumptions on the Hamiltonian and the nonlocal couplings follow closely \cite{bc,ch}.  Some possible generalizations to MFG with local couplings are discussed, based on ideas from \cite{cirant2021maximal}. The main advantage of \textbf{SPI}s vis-\`a-vis \textbf{FP} is that they solve the MFG model in a single iteration loop, since the HJB equation at each stage is already linearized. A minor difference from \cite{bc,ch} is that we prove all results with a learning rate corresponding to an average with increasingly large weights on more recent updates, rather than using simple average. For numerical methods, we use implicit finite difference schemes which are reminiscent of  those from \cite{ad,acd}. Following \cite{ccg}, we use the Engquist-Osher flux for the $\rm{div}(\cdot)$ term in the FPK equation in order to retain the adjoint structure to the discretized (and linearized) HJB equation. The key idea and difference from policy iteration in \cite{ccg,ct} is using some forms of weighted average policies $\bar q^{(n)}$ or $\hat q^{(n)}$, in order to penalize deviation from the previous step and thereby stabilize the learning procedures. This methodology is also reflected in algorithms used for MFG reinforcement learning, e.g. TRPO \cite{guo2023} and online mirror descent \cite{Hadik,perolatOMD}. For the main results we consider coupling terms which are globally Lipschitz, whereas the couplings are (local) uniformly bounded in \cite{ccg}. Numerical results show that convergence of \textbf{SPI}s are sublinear. Importantly, convergences of \textbf{SPI}s are global, i.e. from any initial guesses. This fact is clear from theoretical analysis and supplemented by numerical examples. \par
The paper is organized as follows: In Section \ref{Preliminaries and Assumptions} we give the main notation, functional space setting, assumptions on the data and some preliminary results. We propose two different \textbf{SPI}s in Subsection \ref{Algo1} and \ref{Algo2}. The weighted average policies $\bar q^{(n)}$ or $\hat q^{(n)}$ are shown to converge uniformly to $H_p(x,\nabla u)$ if the potential MFG \eqref{MFG} has a unique classical solution. In Subsection \ref{Stable solution} we show the local convergence of  $(u^{(n)},m^{(n)})$, constructed by a \textbf{SPI}, to a stable solution $(u,m)$ if the system \eqref{MFG} is nonconvex. In Section \ref{Numerical}, we present some finite difference schemes for implementing \textbf{SPI}s with numerical examples.
\section{Preliminaries and Assumptions}\label{Preliminaries and Assumptions}
\subsection{Functional spaces and data}
We first introduce some functional spaces. Given a Banach space $X$, $L^r(0,T;X)$ denotes the usual vector-valued Lebesgue space. For any $r\geq1$, we denote by $W^{1,2}_r(Q)$ the space of functions $u$ such that $\partial_t^{\delta}D^{\zeta}_x u\in L^r(Q)$ for all multi-indices $\zeta$ and $\delta$, $\vert \zeta \vert+2\delta\leq  2$, endowed with the norm
$$
	\|u\|_{W^{1,2}_r(Q)}=\Big(\int_{Q}\sum_{\vert \zeta \vert+2\delta\leq2}\vert \partial_t^{\delta}D^{\zeta}_x u\vert ^rdxdt\Big)^{\frac1r},
$$
with its trace space given by the fractional Sobolev class $W^{2-\frac{2}{r}}_r(\T^d)$. $W^{0,1}_r(Q)$ is endowed with the norm
$
\|u\|_{W^{0,1}_r(Q)}:=\|u\|_{L^r(Q)}+\sum_{\vert \zeta\vert=1}\|D^\zeta_xu\|_{L^r(Q)}.
$
Let $1/r+1/r'=1$, $\mathcal H^1_r(Q)$ denotes the space of functions $u\in W^{0,1}_r(Q)$ and $\partial_tu\in (W^{0,1}_{r'}(Q))'$, endowed with the norm
$
\|u\|_{\mathcal H^1_r(Q)}:=\|u\|_{W^{0,1}_r(Q)}+\|\partial_tu\|_{(W^{0,1}_{r'}(Q))'}.
$
Let $\mathcal C^0(Q)$ denotes the space of continuous functions in $Q$ and $\mathcal C^{0,1}(Q)$ the space of functions in $\mathcal C^0(Q)$ and once continuously differentiable w.r.t. $x$. For $\alpha \in (0,1)$, $\mathcal C^{\alpha/2,\alpha}(Q)$, $\mathcal C^{(1+\alpha)/2,1+\alpha}(Q)$, $\mathcal C^{1+\alpha/2,2+\alpha}(Q)$ and $\mathcal C^{2+\alpha}(\T^d)$ denote the spaces of H\"older continuous functions in $Q$ and on $\T^d$, with norms defined as in e.g. \cite[p. 5]{achdou2015}. $\mes(\T^d)$ is the set of Borel probability measures on $\T^d$, with finite first moments and endowed with the Wasserstein distance: for $m,m' \in \mes(\T^d)$, ${\bf{d}}_1(m,m')=\sup_\phi \int_{\T^d}d(m-m')(x)$ where the supremum is taken over all $1$-Lipschitz maps $\phi: \T^d\rightarrow \R$.
Given a map $\mathcal{U}:\mes(\states)\to \R$, $\frac{\delta \mathcal{U}}{\delta m}:\states\times \mes(\states)\to \R$ denotes the flat derivative of $\mathcal{U}$ if:
$$
\mathcal{U}[m']-\mathcal{U}[m]= \int_0^1\mint \frac{\delta \mathcal{U}}{\delta m}[(1-s)m+sm'](x)\d (m'-m)(x) \d s, 
$$  
with the normalization $\int_{\T^d} \frac{\delta \mathcal{U}[m]}{\delta m}(x)\d m(x)=0$. Higher order derivatives are defined similarly.\par
Throughout the paper, we assume $f,g : \states \times \mes(\states) \to \R$ derive from potentials, i.e., there exists $F,G : \mes(\states) \to \R$,
\begin{equation}\label{Potential}
 \frac{\delta F}{\delta m}(x)= f[m](x), \,\, \frac{\delta G}{\delta m}(x)= g[m](x).
\end{equation}

We now state some key assumptions on the data.
 \begin{itemize}
          \item[($A1$)] The initial condition $m_0\in \Pk \cap \mathcal C^{2+\alpha}(\T^d)$ and $m_0(x)\geq \vartheta >0$, $\forall x\in \T^d$. 
	\item[($A2$)] For all $x\in \T^d$, $p\in \R^d$ and some $\bar C>0$: 
\begin{equation}\label{hyp:unifCv}
\begin{gathered}
H\in \mathcal C^2(\T^d\times \R^d;\R)\,\,\text{and}\,\,\frac{1}{\bar C} I_d\leq H_{pp}(x,p) \leq \bar C I_d,\\
|H_{px}(x,p)|\leq \bar C(|p| + 1), \,\,\,|H_{xx}(x,p)| \leq \bar C (|p|^2 + 1).
\end{gathered}
\end{equation}	
          \item[($A3$)]
$f$, $g$ $:\T^d\times \mes (\T^d)\rightarrow \R$ and their space derivatives $\partial_{x_i}f$, $\partial_{x_i}g$, $\partial_{x_ix_j}g$ are all globally Lipschitz continuous. The measure derivatives $\frac{\delta f}{\delta m}$ and $\frac{\delta g}{\delta m}$$:\T^d\times \mes (\T^d)\times \T^d\rightarrow \R$ are also Lipschitz continuous.
\item[($A4$)]   For any $m,m'\in \mes(\states)$, 
\begin{equation}\label{mono}
\begin{split}
\int_{\T^d} \left(f[m](x)-f[m'](x)\right)d(m-m')(x)\geq 0, \\
\int_{\T^d} \left(g[m](x)-g[m'](x)\right)d(m-m')(x)\geq 0.
\end{split}
\end{equation}
\end{itemize}
We can also replace $f[m(t)](x)$ in \eqref{MFG} by the local coupling $\tilde{f}(m(t,x))$ in the sense that $\tilde{f}(\cdot): \R^+\rightarrow\R$ with the following assumption:
\begin{itemize}
\item[($A5$)] $u(T,x)=g_T(x)$. There exist $\varkappa>1$ and $C_{\tilde{f}}>0$ such that either
\begin{equation}\label{f+}\tag{f+}
C_{\tilde{f}}^{-1}m^{\varkappa-1}\leq \frac{\partial \tilde{f}}{\partial m}\leq C_{\tilde{f}}(m^{\varkappa-1}+1),
\end{equation}
\begin{equation}\label{f-}\tag{f-}
\text{or}\,\,-C_{\tilde{f}}(m^{\varkappa}+1)\leq \tilde{f}(m) \leq C_{\tilde{f}}(m^{\varkappa}+1),
\end{equation}
with $\varkappa<\frac{d}{d-2}$ if $d>2$ and no further restrictions on $\varkappa$ for $d=1,2$.
 \end{itemize}
($A4$) can be interpreted as crowd aversion, i.e. incentivizing agents to disperse. It is clear under assumption $(A5)$ with \eqref{f+}, we obtain a local coupling version of ($A4$):
$\int_{\T^d} (\tilde{f}(m)-\tilde{f}(m'))d(m-m')(x)\geq 0$. The potential structure in the local coupling case can be formulated as
$
\tilde F(m)=\int_0^m\tilde f(\tau)d\tau$, if $\tilde f(0)=0$.
\begin{remark}\label{Existence}
We recall from \cite{ll}, under assumptions ($A1$), ($A2$) and ($A3$), the system \eqref{MFG} has at least one classical solution $(u,m)\in \mathcal C^{1+\alpha/2,2+\alpha}(Q)\times \mathcal C^{1+\alpha/2,2+\alpha}(Q)$. The solution is unique if, in addition, ($A4$) holds. As special cases of Theorem 1.4 and Theorem 1.5 from \cite{cirant2021maximal}, under assumptions ($A1$), ($A2$) and ($A5$) with \eqref{f+}, there exists a unique classical solution to \eqref{MFG}. If ($A5$) is with \eqref{f-} then there exists at least one classical solution to \eqref{MFG}.
\end{remark}
\begin{remark}\label{F f Lip} From assumption ($A3$), there exists a constant $C>0$,
\begin{align}
\sup_{x\in \T^d}\big|f[m'](x)-f[m](x)\big|+\sup_{x\in \T^d}\big|\partial_xf[m'](x)-\partial_xf[m](x)\big|&\leq C {\bf{d}}_1(m,m'),\label{f Lip}\\
\big|F[m' ] - F[m] - \int_{\T^d}  f[m](x) \d (m'-m)(x) \big| &\leq C {\bf{d}}^2_1(m,m') ,\label{F Lip}\\
\big|G[m'] - G[m]-\int_{\T^d}  g[m](x) \d (m'-m)(x) \big| &\leq C {\bf{d}}^2_1(m,m'),\label{G Lip}\\
\sup_{x\in\T^d}\big|f[m' ](x) - f[m](x) - \frac{\delta f}{\delta m}[m](x)(m'-m)\big| &\leq C{\bf{d}}^2_1(m,m') ,\label{delta f Lip}\\
\sup_{x\in\T^d}\big|g[m' ](x) - g[m](x) - \frac{\delta g}{\delta m}[m](x)(m'-m)\big| &\leq C{\bf{d}}^2_1(m,m') .\label{delta g Lip}
\end{align}
In particular, ${\bf{d}}_1(m,m')\leq d\|m-m'\|_{\mathcal C^0(\T^d)}$ if $m,m'\in \mes(\T^d)\cap \mathcal C^0(\T^d)$.
\end{remark}
\begin{remark}\label{Bernstein}
The key to obtain existence results for system \eqref{MFG} is to have an \textit{a priori} Bernstein estimate $\|\nabla u\|_{L^\infty(Q;\R^d)}\leq R_0$, with a constant $R_0$ generally not explicitly tractable but only depends on the data of the problem. This would allow us to obtain $\|H_p(x,\nabla u)\|_{L^\infty(Q;\R^d)}\leq R$ for some $R>0$. Under assumptions $(A1)$,$(A2)$ and $(A3)$ the Bernstein estimate is straight forward from \cite[Theorem 1.3]{cirant2020lipschitz}, since $\partial_x(f[m(t)](\cdot))$ has uniformly bounded $\mathcal C^2$ norms in the space variable.\par
In the local coupling case with $(A1)$, $(A2)$ and $(A5)$ the Bernstein estimate is much more involved as the regularity of $\partial_x\tilde{f}(m)$ depends on $m$, but it has been obtained in \cite[Theorem 1.4]{cirant2021maximal} by adjoint methods.
\end{remark}
\subsection{Potential MFG and stable solutions}
We can consider the potential formulation, for $t_0 \in [0,T]$,
\begin{equation}\label{J}
J_{t_0}(m,w):=\int_{t_0}^T\int_{\T^d}L(x,\frac{w}{m})mdxdt+\int_{t_0}^TF[m(t)]dt+G[m(T)],
\end{equation}
such that $m(0,x)=m_0$ and $(m,w)$ satisfy the  equation
$$
\partial_t m- \sigma \Delta m - \text{div} ( mq) =0, \,\,{\rm in}\,\,Q.
$$
Denote $q=\frac{w}{m}$ if $m>0$. Any minimizer $(m,w)$ of $J_0$ corresponds to a solution to the MFG system \eqref{MFG}, in the sense that a pair $(u,m)$ solves \eqref{MFG} and $q=H_p(x,\nabla u)$. It was shown in \cite[Proposition 3.2]{bc} that for any later time $t_0>0$, $(m,w)$ is also the unique minimizer of $J_{t_0}$. For MFG systems which may have nonconvex potentials and therefore multiple solutions, Cardaliaguet and Briani introduced in \cite[Definition 4.1]{bc} the notion of stable solutions as local attractors for the learning procedure of the \textbf{FP} algorithm. 
  \begin{definition}
Let $(u,m)$ be a solution the MFG system \eqref{MFG} with initial condition $(t_0,m_0)\in [0,T]\times \mathcal{P}(\T^d)$. We say the solution is stable if $(v,\mu)=(0,0)$ is the unique solution to the linearized system
\be\label{Linear MFG}
\left\{\begin{split}
(i) &-\partial_t v - \sigma \Delta v  +H_p(x,\nabla u)\cdot \nabla v =\frac{\delta f}{\delta m}[m(t)](x)(\mu)& {\rm in } \,\,[t_0,T]\times \T^d, \\
(ii) &\partial_t \mu- \sigma \Delta \mu  -{\rm{div}}(\mu H_p(x,\nabla u) )-{\rm{div}}( mH_{pp} (x,\nabla u)\nabla v) =0 & {\rm in } \,\,[t_0,T]\times \T^d,  \\
&\mu(t_0,x)=0 , \; v(T, x)=\frac{\delta g}{\delta m}[m(T)](x)(\mu(T))& {\rm in } \,\,\T^d.
\end{split}\right.
\ee
\end{definition}
 Stable solutions are shown in \cite[Proposition 4.2]{bc} to be locally isolated in the sense that, there exists an $\eta$-ball around $(u,m)$, such that no other solutions exist within this ball. The following estimate (\cite[Lemma 5.2]{bc}) is essential to consider convergence to stable solutions:
\begin{lemma}\label{stable sol lemma}
Let $(u,m)$ be a stable solution the MFG system \eqref{MFG}. Then there exists a constant $C>0$ such that, for any $a\in \mathcal C^0(Q)$, $b\in \mathcal C^0(Q;\R^d)$, $c\in \mathcal C^0(\T^d)$ and $(v,\rho)$ solution to 
\begin{equation}\label{Linerized}
\left\{\begin{split}
(i) &-\partial_t v - \sigma \Delta v  +H_p(x,\nabla u)\cdot \nabla v =\frac{\delta f}{\delta m}[m](x)(\rho) +a(t,x)& {\rm in } \,\,Q,\\
(ii) &\partial_t \rho- \sigma \Delta \rho  -{\rm{div}}(\rho H_p(x,\nabla u) )-{\rm{div}}( mH_{pp} (x,\nabla u)\nabla v) ={\rm{div}} (b(t,x))& {\rm in } \,\,Q,  \\
&\rho(0,x)=0 , \; v(T, x)=\frac{\delta g}{\delta m}[m(T)](x)(\rho(T))+c(x)& {\rm in } \,\,\T^d.
\end{split}\right.
\end{equation}
one has 
$$
 \|v\|_{\mathcal C^{0,1}(Q)}+\|\rho\|_{\mathcal C^{0}(Q)} \leq C\big(\|a\|_{\mathcal C^{0}(Q)}+\|b\|_{\mathcal C^{0}(Q;\R^d)}+\|c\|_{\mathcal C^{0}(\T^d)}  \big).
$$
\end{lemma}
\subsection{Some results from convex analysis}
The Lagrangian $L$ is the  convex conjugate of $H$ in \eqref{defH}: 
\begin{equation}\label{L H}
L(x,q)= \sup_{p\in \R^d} \; p\cdot q - H(x,p).
\end{equation}
\begin{remark}\label{H R}
Under assumption $(A2)$, it is known (e.g. \cite[Corollary A.2.7]{semiconcave}) that for $p\in \R^d$ there exists a unique value $q(x,p)=H_p(x,p)$ such that
$
H_{pp}(x,p)=\frac{1}{L_{qq}(x,q(x,p))},\,\,\text{hence}\,\,\frac{1}{\bar C} I_d\leq L_{qq}(x,q) \leq \bar C I_d.
$
 We use the index $\iota \in \{\iota_1,\iota_2\}$ such that $\| q^{(\iota)}\|_{L^\infty(Q;\R^d)}\leq R$, then there exists a constant $C$ depending only on $\bar C$ and $R$:
\begin{equation}\label{L Lipschitz}
\| L(x,q^{(\iota_1)})-L(x,q^{(\iota_2)})\|_{L^\infty(Q)} \leq C\| q^{(\iota_1)}-q^{(\iota_2)}\|_{L^\infty(Q;\R^d)}.
\end{equation}
\end{remark}
\begin{lemma}\label{H-L}
Assume $(A2)$ for $H$, $|p|\leq R_0$ and $\vert H_p(x,p)\vert \leq R$ with $p\in \R^d$. Then for all  $q\in \R^d$ such that $\vert q\vert \leq R$, there exists a constant $C$ depending only on $\bar C$ and $R$:
$$\frac{1}{2\bar C} \leftnorm q-H_p(x,p) \rightnorm^2 \leq H(x,p) + L(x,q) - \lg p,q\rg \leq C\leftnorm q-H_p(x,p) \rightnorm .$$
\end{lemma}
\begin{proof} From \eqref{L H} and \eqref{L Lipschitz} we obtain
$$
H(x,p) + L(x,q) - \lg p,q\rg=\lg p,H_p(x,p)-q\rg+L(x,q)-L(x,H_p(x,p))\leq C\leftnorm q-H_p(x,p) \rightnorm.
$$
For the other side of the inequality we refer to \cite[Lemma 2.2]{ch}.
\end{proof}\par
Finally, we give a purely technical result from mathematical analysis. It plays a pivotal role in the theory of Fictitious Play and has been proved in \cite{MondrerShapley2} and \cite[Lemma 2.7]{ch}.
\begin{lemma}\label{a_n}
Consider a sequence of positive real numbers $\{ a_n \}_{n \in \N}$ such that $\sum_{n=1}^{\infty} a_n / n < +\infty$. Then we have ${\lim_{N \to \infty} \frac{1}{N} \sum_{n=1}^{N} a_n = 0}$.
In addition, if there is a constant $C >0 $ such that $\abs a_n - a_{n+1} \abs < C/n$ then $\lim_{n \to \infty} a_n = 0$.
\end{lemma}\par
\section{Main results}
We introduce two Smoothed Policy Iteration algorithms, referred to as algorithm \textbf{SPI1} and \textbf{SPI2}. Throughout this section, unless otherwise specified, $C$ is used to denote a generic positive constant which only depends on the data of the problem ($H$, $f$, $g$ and $m_0$) and may increase from line to line. 
\subsection{Algorithm SPI1}\label{Algo1}
Initialize a vector field   $q^{(0)}$, $q^{(0)}\in \mathcal C^{\alpha/2,\alpha}(Q;\R^d)$ and $\| q^{(0)}\|_{L^\infty(Q;\R^d)}\leq R$. Let $q^{(0)}=\bar  q^{(0)}$ and iterate for each $n\geq 0$:
\begin{itemize}
	\item[\textbf{(i)}] \textbf{Generate the distribution from the current policy}. Solve
	\begin{equation}\label{m update}
		\left\{\begin{split}
			&\partial_t m^{(n)}-\sigma \Delta m^{(n)}-\text{div} (m^{(n)} \bar{q}^{(n)})=0,\qquad &\text{ in }Q,\\
			& m^{(n)}(0,x)=m_0(x)&\text{ in }\T^d.
		\end{split}\right.
	\end{equation}
	\item[\textbf{(ii)}] \textbf{Policy evaluation}. Solve
	\begin{equation}\label{u update}
		\left\{\begin{split}
			&-\partial_t u^{(n)}- \sigma \Delta u^{(n)}+\bar{q}^{(n)} \cdot \nabla u^{(n)}-L(x,\bar{q}^{(n)})=f[m^{(n)}](x) &\text{ in }Q,\\
			&u^{(n)}(T,x)=g[m ^{(n)} (T)](x)&\text{ in }\T^d.
		\end{split}\right.
	\end{equation}
	\item[\textbf{(iii)}] \textbf{Policy update}.
	\begin{equation}\label{q update}
		q^{(n+1)}(t,x)={\arg\max}_{q\in \R^d,\,\vert q\vert \leq R}\left\{q\cdot \nabla u^{(n)}(t,x)-L(x,q)\right\}\qquad\text{ in }Q.
	\end{equation}
       \item[\textbf{(iv)}] \textbf{Smoothing}.
         Iteration terminates if $\|q^{(n+1)}-q^{(n)}\|_{L^\infty(Q;\R^d)}$ is small enough, else
	\begin{equation}\label{q bar update}
	\bar{q}^{(n+1)}=(1-\frac{2}{n+2})\bar{q}^{(n)}+\frac{2}{n+2}q^{(n+1)}.
	\end{equation}
\end{itemize}
The interpretation is that, in each stage $n$, the agent holds the belief that other agents are acting according to the smoothed policy $\bar{q}^{(n)}$, hence she generates from $\bar{q}^{(n)}$ the probability distribution $m^{(n)}$. She then evaluates $\bar{q}^{(n)}$ by solving for the value function $u^{(n)}$. She finds a greedy update $q^{(n+1)}$ based on $u^{(n)}$ and use it to update $\bar{q}^{(n+1)}$. \par
We denote: $\delta u^{(n+1)}=u^{(n+1)}-u^{(n)},\,\,\delta m^{(n+1)}=m^{(n+1)}-m^{(n)}$.
\begin{remark}\label{learning rate}
We can interpret the learning rate by rewriting \eqref{q bar update} as, $\forall n\geq 1$,
 \begin{equation}\label{learning rate sum}
\bar{q}^{(n)}=\frac{1}{\sum_{k=1}^{n}k}\sum_{k=1}^n kq^{(k)}.
	\end{equation}
	As an agent plays the game repeatedly, she puts increasingly larger weights in more recent observations of $q^{(n)}$ while updating the ``belief" $ \bar{q}^{(n)}$, cf. \cite[p. 191]{deschamps1975}.
	It is to be noted that all the following results in this subsection hold for \textbf{SPI1} if we replace the learning rate $2/(n+2)$ by $1/(n+2)$, i.e. $\forall n\geq 1$, 
	$$
\tilde{q}^{(n+1)}=(1-\frac{1}{n+2})\tilde{q}^{(n)}+\frac{1}{n+2}q^{(n+1)}=\frac{1}{n+2}\sum_{k=0}^{n+1} q^{(k)}.
	$$
	We focus on the $2/(n+2)$ learning rate as it speeds up the convergence from a numerical point of view, for comparison with learning rates in \textbf{FP} we refer to \cite{blp}.
	\end{remark}
	\begin{remark}\label{check}
We use an a posterior argument to check that $(u^{(n)},m^{(n)})$ terminates at a solution to \eqref{MFG}, if ideally, $\|q^{(N+1)}-q^{(N)}\|_{L^\infty(Q;\R^d)}=0$. \eqref{q update} implies then $\|\bar q^{(N+1)}-\bar q^{(N)}\|_{L^\infty(Q;\R^d)}=0$. From \cref{m stability} we obtain $\|m^{(N+1)}-m^{(N)}\|_{\mathcal H^1_r(Q)}=0$ and then $\|u^{(N+1)}-u^{(N)}\|_{W^{1,2}_r(Q)}=0$ from \cref{linear estim Sobolev}, hence $\|\nabla u^{(N+1)}-\nabla u^{(N)}\|_{L^\infty(Q;\R^d)}=0$ and $q^{(N)}=H_p(x,\nabla u^{(N)})$ pointwise. Morevoer, by induction $\|q^{(N+i)}-q^{(N)}\|_{L^\infty(Q;\R^d)}=0$ for all $i\geq 1$, i.e. $\{ q^{(n)}\}$ converges to $q^{(N)}$ uniformly. Hence using the expression \eqref{learning rate sum} and Silverman$-$Toeplitz theorem, $\{ \bar q^{(n)}\}$ converges to $q^{(N)}$ uniformly. Therefore, $(u^{(N)},m^{(N)})$ is a solution to \eqref{MFG}. This justifies using $\|q^{(N+1)}-q^{(N)}\|_{L^\infty(Q;\R^d)}$ small enough for the stopping criteria.
\end{remark}
\begin{lemma}\label{Solution u m1}
For each $n$, $n\geq 1$, there exists a unique solution 
$
\big(u^{(n)},m^{(n)}\big)\in \mathcal C^{1+\alpha/2,2+\alpha}(Q)\times  \mathcal C^{1+\alpha/2,2+\alpha}(Q)
$
 to the system \eqref{m update}-\eqref{u update}. Moreover, we have $ m^{(n)}\geq 1/C>0$ and
 \begin{equation}
\begin{split}
\|u^{(n)}\| _{\mathcal C^{1+\alpha/2,2+\alpha}(Q)}+\| m^{(n)}\| _{\mathcal C^{1+\alpha/2,2+\alpha}(Q)}&\\
+\|q^{(n)}\|_{\mathcal C^{\alpha/2,\alpha}(Q;\R^d)}
+\| \bar q^{(n)}\|_{\mathcal C^{\alpha/2,\alpha}(Q;\R^d)}&\leq C.
\end{split}
\end{equation}
\end{lemma}
\begin{proof}
We use a bootstrap argument. During the proof all bounds are independent of $n$. From $\| q^{(n)}\|_{L^\infty(Q;\R^d)} \leq R$, from \eqref{q bar update}, $(A3)$ and \cref{H-L}, we have:
\begin{equation*}
\begin{split}
\| \bar q^{(n)}\|_{L^\infty(Q;\R^d)}+\|L(\cdot,\bar q^{(n)})\|_{L^\infty(Q;\R^d)}&\\
 +\sup_{t\in [0,T]}\|f[m^{(n)}(t)](\cdot)\|_{\mathcal C^{2}(\T^d)}+\| g[m^{(n)}(T)](\cdot)\|_{\mathcal C^3(\T^d)}&\leq C,\,\,\,\forall n\geq 0.
\end{split}
\end{equation*}             
From \cref{linear estim Sobolev}, $\| u^{(n)}\| _{W^{1,2}_r(Q)}$ is bounded for all $r\in(d+2,+\infty)$, hence
$u^{(n)}$ is bounded in $\mathcal C^{1-\frac{d+2}{2r},2-\frac{d+2}{r}}(Q)$ and $\nabla u^{(n)}$ is bounded in $\mathcal C^{\alpha/2,\alpha}(Q;\R^d)$. With $(A2)$, we obtain that $q^{(n+1)}=H_p(x,\nabla u^{(n)})$ is bounded in $\mathcal C^{\alpha/2,\alpha}(Q;\R^d)$. Then for $n\geq 0$, by \eqref{q bar update}, $\bar q^{(n)}$ is bounded in $\mathcal C^{\alpha/2,\alpha}(Q;\R^d)$. From \cref{H-L}, $L(\cdot,\bar{q}^{(n)})$ is bounded in $\mathcal C^{\alpha/2,\alpha}(Q)$. For $n\geq 0$, $\|u^{(n)}\| _{\mathcal C^{1+\alpha/2,2+\alpha}(Q)}\leq C$ follows from \eqref{u update} and \cref{linear estim}. This and  ($A2$) allow us to obtain $\|{\rm{div}}H_p(\cdot,\nabla u^{(n)})\|_{\mathcal C^{\alpha/2,\alpha}(Q;\R^d)}\leq C$, $n\geq 1$. Then we have $\|{\rm{div}}\bar{q}^{(n)}\|_{\mathcal C^{\alpha/2,\alpha}(Q;\R^d)}\leq C$, $n\geq 1$ from \eqref{q bar update}. For $n\geq 1$, $\| m^{(n)}\| _{\mathcal C^{1+\alpha/2,2+\alpha}(Q)}\leq C$ follows from \cref{linear estim}. The bound $m^{(n)}\geq 1/C$ follows from \cref{m stability} (iii). 
\end{proof}

\begin{lemma}\label{C/n bound algo1}
There exist a constant $C$, such that for all $n\geq 1$,
\begin{equation}
\begin{split}
\| m^{(n+1)}-m^{(n)}\|_{\mathcal C^{\alpha/2,\alpha}(Q)}+\| u^{(n+1)}-u^{(n)}\| _{\mathcal C^{1+\alpha/2,2+\alpha}(Q)}&\\
+\| q^{(n+1)}- q^{(n)}\|_{\mathcal C^{\alpha/2,\alpha}(Q;\R^d)}&\leq \frac{C}{n}.
\end{split}
\end{equation}
\end{lemma}
\begin{proof}
We can obtain from \eqref{q bar update} and \cref{Solution u m1} that for all $n\geq 1$,
\begin{equation}\label{barq-barq}
\| \bar q^{(n+1)}- \bar q^{(n)}\|_{\mathcal C^{\alpha/2,\alpha}(Q;\R^d)} \leq \frac{2}{n+2}\|q^{(n+1)}- \bar q^{(n)}\|_{\mathcal C^{\alpha/2,\alpha}(Q;\R^d)}\leq \frac{C}{n}.
\end{equation}
It follows from \cref{m stability} and \cref{H embedding} that $\|\delta m\|_{\mathcal C^{\alpha/2,\alpha}(Q)}\leq C/n$. From \eqref{u update} we have 
\begin{equation}\label{U}
\begin{split}
&-\partial_t\delta u^{(n+1)}-\sigma \Delta \delta u^{(n+1)}+\bar q^{(n+1)}\cdot \nabla \delta u^{(n+1)}\\
={}&f[m^{(n+1)}](x)-f[m^{(n)}](x)-\nabla u^{(n)}(\bar q^{(n+1)}-\bar q^{(n)})+L(x,\bar q^{(n+1)})-L(x,\bar q^{(n)}),
\end{split}
\end{equation}
where $\delta u(T,x)=g[m^{(n+1)} (T)](x)-g[m^{(n)} (T)](x)$. From \cref{F f Lip},
$$
\|f[m^{(n+1)}](\cdot)-f[m^{(n)}](\cdot)\|_{\mathcal C^{\alpha/2,\alpha}(Q)}+\|\delta u(T,\cdot)\| _{\mathcal C^{2+\alpha}(\T^d)}\leq \frac{C}{n}.
$$
We can obtain from \eqref{L Lipschitz} and \cref{Solution u m1} that
$
\|-\nabla u^{(n)}(\bar q^{(n+1)}-\bar q^{(n)})+L(\cdot,\bar q^{(n+1)})-L(\cdot,\bar q^{(n)})\|_{\mathcal C^{\alpha/2,\alpha}(Q)}\leq \frac{C}{n}.
$
$\| \delta u^{(n+1)}\| _{\mathcal C^{1+\alpha/2,2+\alpha}(Q)}\leq \frac{C}{n}$ follows from \cref{linear estim}.
\end{proof}
\begin{theorem}\label{Maintheorem1} Under the assumptions (A1), (A2), (A3) and with $R$ sufficiently large, the family $\{ (u^n , m^n) \}$, $n \in \N$, in the algorithm \textbf{SPI1} is uniformly continuous and any cluster point is a solution to the  second order MFG \eqref{MFG}. If, in addition, the monotonicity condition (A4) holds, then the whole sequence converges to the unique solution of  \eqref{MFG}. 
\end{theorem}
\begin{proof}
From the Bernstein estimate (\cref{Bernstein}) there exists a sufficiently large $R$ such that $$\|H_p(\cdot,\nabla u)\|_{L^\infty(Q;\R^d)}\leq R. $$We  show that $\{H_p(x,\nabla u^{(n)})- \bar q^{(n)}\}$ converges to $0$ uniformly. From the convexity of $L(x,q)$ in the $q-$variable, we have 
 \begin{equation*}\label{L convexity}
 L(x,\bar q^{(n+1)})\leq \frac{n}{n+2}L(x,\bar q^{(n)})+\frac{2}{n+2}L(x, q^{(n+1)}).
\end{equation*}
We can then obtain from equations \eqref{q update} and \eqref{q bar update} ,
\begin{equation*}
\begin{split}
&\frac{2}{n+2} H(x,\nabla u^{(n)})\\
= {} &\frac{2}{n+2} q^{(n+1)}\cdot \nabla u^{(n)}-\frac{2}{n+2}L(x, q^{(n+1)})\\
= {} &\big(\bar{q}^{(n+1)}-\frac{n}{n+2}\bar q^{(n)}\big)\cdot \nabla u^{(n)}-\frac{2}{n+2}L(x, q^{(n+1)})\\
\leq {} &-\frac{n}{n+2}\bar q^{(n)}\cdot \nabla u^{(n)}+ \bar{q}^{(n+1)} \cdot \nabla u^{(n)}+\frac{n}{n+2}L(x,\bar q^{(n)})-L(x,\bar q^{(n+1)}).
\end{split}
\end{equation*}
We can then obtain from \eqref{u update},
\begin{equation}\label{1/n+2 H 1}
\begin{split}
&\frac{2}{n+2} H(x,\nabla u^{(n)})-\frac{2}{n+2}\big(\bar q^{(n)}\cdot \nabla u^{(n)}-L(x,\bar q^{(n)}\big)\\
\leq {} &  -\bar{q}^{(n+1)} \cdot \nabla \delta u^{(n+1)}+\bar{q}^{(n+1)} \cdot \nabla u^{(n+1)}-L(x,\bar q^{(n+1)})-\bar{q}^{(n)} \cdot \nabla u^{(n)}+L(x,\bar q^{(n)})\\
\leq {} &  -\bar{q}^{(n+1)} \cdot \nabla \delta u^{(n+1)}+\partial_t \delta u^{(n+1)}+\sigma \Delta \delta u^{(n+1)}+f[m^{(n+1)}](x)-f[m^{(n)}](x)
\end{split}
\end{equation}
We can obtain from \cref{H-L}, \cref{Solution u m1} and Cauchy Schwartz inequality that
\begin{equation}\label{H-barq}
\begin{split}
 &\frac{1}{\bar{C}n}|H_p(x,\nabla u^{(n)})-\bar{q}^{(n)}|^2\\
 =  {} & \frac{2}{2\bar{C}(n+2)}|H_p(x,\nabla u^{(n)})-\bar{q}^{(n)}|^2+\frac{2|H_p(x,\nabla u^{(n)})-\bar{q}^{(n)}|^2}{\bar{C}n(n+2)}\\
\leq  {} & \frac{2}{2\bar{C}(n+2)}|H_p(x,\nabla u^{(n)})-\bar{q}^{(n)}|^2+\frac{4|H_p(x,\nabla u^{(n)})|^2+4|\bar{q}^{(n)}|^2}{\bar{C}n(n+2)}\\
\leq  {} & \frac{2}{2\bar{C}(n+2)}|H_p(x,\nabla u^{(n)})-\bar{q}^{(n)}|^2+\frac{C}{n^2}\\
\leq {}& \frac{2}{n+2} H(x,\nabla u^{(n)})-\frac{2}{n+2} \big(\bar{q}^{(n)} \cdot \nabla u^{(n)}-L(x,\bar{q}^{(n)})\big)+\frac{C}{n^2}\\
\leq  {} & \partial_t \delta u^{(n+1)}+\sigma \Delta \delta u^{(n+1)}-\bar{q}^{(n+1)} \cdot \nabla \delta u^{(n+1)}+f[m^{(n+1)}](x)-f[m^{(n)}](x)+\frac{C}{n^2},
\end{split}
\end{equation}
where the last inequality is obtained from \eqref{1/n+2 H 1}. 
 From integration by parts and \eqref{m update},
\begin{equation*}
\begin{split}
&\int_Qm^{(n+1)}\left(\partial_t \delta u^{(n+1)}+\sigma \Delta \delta u^{(n+1)}-\bar{q}^{(n+1)} \cdot \nabla \delta u^{(n+1)}\right)dxdt\\
={}&\int_{\T^d}m^{(n+1)}(T)\big(g[m^{(n+1)}(T)](x))-g[m^{(n)}(T)](x)\big)dx-\int_{\T^d}m_0(x)\delta u^{(n+1)}(0,x)dx.
\end{split}
\end{equation*}
We multiply both sides of \eqref{H-barq} by $m^{(n+1)}$, integrate on $Q$ to obtain
\begin{align*}
\begin{split}
&\frac{1}{\bar{C}}\int_Q \frac{1}{n}m^{(n+1)}|H_p(x,\nabla u^{(n)})-\bar{q}^{(n)}|^2dxdt\\
\leq {}&\int_Q \left(m^{(n+1)}f[m^{(n+1)}](x)-m^{(n)}f[m^{(n)}](x)\right)dxdt-\int_Qf[m^{(n)}](x)\delta m^{(n+1)}dxdt\\
&+ \int_{\T^d}\left(m^{(n+1)}(T)g[m^{(n+1)}(T)](x)-m^{(n)}(T)g[m^{(n)}(T)](x)\right)dx\\
&-\int_{\T^d}g[m^{(n)}(T)](x)\delta m^{(n+1)}(T,x)dx-\int_{\T^d}m_0(x)\delta u^{(n+1)}(0,x)dx+\frac{C}{n^2},
 \end{split}
 \end{align*}
where we used the trivial identity:
 $$
m^{(n+1)}\left( f[m^{(n+1)}]-f[m^{(n)}]\right)  = m^{(n+1)}f[m^{(n+1)}]-m^{(n)}f[m^{(n)}]-f[m^{(n)}]\delta m^{(n+1)}.
 $$
  We can then use \eqref{F Lip} and \eqref{G Lip} to obtain
 \begin{equation}
 \begin{split}
 &\frac{1}{\bar{C}}\int_Q \frac{1}{n}m^{(n+1)}|H_p(x,\nabla u^{(n)})-\bar{q}^{(n)}|^2dxdt\\
 \leq  {} &\int_0^T\left(F[m^{(n)}(t)]-F[m^{(n+1)}(t)]\right)dt+G[m^{(n)}(T)]-G[m^{(n+1)}(T)]+\frac{C}{n^2}\\
 {} &+\int_Q \left(m^{(n+1)}f[m^{(n+1)}](x)-m^{(n)}f[m^{(n)}](x)\right)dxdt-\int_{\T^d}m_0(x)\delta u^{(n+1)}(0,x)dx\\
 {} &+\int_{\T^d}\left(m^{(n+1)}(T)g[m^{(n+1)}(T)](x))-m^{(n)}(T)g[m^{(n)}(T)](x)\right)dx.
\end{split}
\end{equation}
Let
\begin{equation}\label{a_n 1}
a_n=\int_Q m^{(n+1)}|H_p(x,\nabla u^{(n)})-\bar{q}^{(n)}|^2dxdt,
\end{equation}
by telescoping we can see that, $\forall N\geq 2$,
\begin{equation*}
 \begin{split}
\frac{1}{\bar C}\sum_{n=1}^{N}\frac{a_n}{n}\leq  {} &\int_0^T\left(F[m^{(1)}(t)]-F[m^{(N+1)}(t)]\right)dt+G[m^{(1)}(T)]-G[m^{(N+1)}(T)]\\
 {} &+\int_Q \left(m^{(N+1)}f[m^{(N+1)}](x)-m^{(1)}f[m^{(1)}](x)\right)dxdt\\
 {}&-\int_{\T^d}m_0(x)\left(u^{(N+1)}(0,x)-u^{(1)}(0,x)\right)dx\\
 {} &+\int_{\T^d}\left(m^{(N+1)}(T)g[m^{(N+1)}(T)](x)-m^{(1)}(T,x)g[m^{(1)}(T)](x)\right)dx+C,
\end{split}
\end{equation*}
hence, $\sum_{n=1}^{\infty} a_n / n < +\infty$. One can obtain from Lemma \ref{Solution u m1} and \ref{C/n bound algo1} that $\abs a_n - a_{n+1} \abs < C/n$. Then we obtain $a_n\rightarrow 0$ from Lemma \ref{a_n}. From \cref{Solution u m1} we have $m^{(n)}(t,x)\geq 1/C$, $\forall (t,x)\in Q$, hence
$
\lim_{n\rightarrow \infty}\int_Q |H_p(x,\nabla u^{(n)})-\bar{q}^{(n)}|^2dxdt=0.
$
From \cref{Solution u m1}, the sequence $\{H_p(x,\nabla u^{(n)})- \bar q^{(n)}\}$ is uniformly continuous, hence converges to $0$ uniformly on $Q$.\par
From \cref{Solution u m1}, the sequence $\{\bar q^{(n)}\}$ is pre-compact for uniform convergence, one can extract a uniformly convergent subsequence $\{\bar q^{(n_i)}\}$. From \cref{m stability} and \cref{H embedding}, $\{m^{(n_i)}\}$ converges uniformly. It then follows from ($A3$) that $\{f[m^{(n_i)}]\}$ converges uniformly. From Proposition \ref{linear estim Sobolev} and the fact  $\{H_p(x,\nabla u^{(n_i)})- \bar q^{(n_i)}\}$ converges to $0$ uniformly, one can pass to the limit in \eqref{m update} and \eqref{u update}. Any cluster point of $\{u^{(n_i)}, m^{(n_i)}\}$, denoted by $(u,m)$, is then a classical solution to the system \eqref{MFG}. 
If $(A4)$ holds, then the MFG system has a unique classical solution, so that the whole sequence $\{(u^{(n)},m^{(n)})\}$ has a unique cluster point $(u,m)$ and thus converges uniformly to $(u,m)$. 
\end{proof}
\begin{remark}\label{f local} Similar results to \cref{Maintheorem1} hold under assumptions ($A1$) ($A2$) and ($A5$) if we replace the $f[m](x)$ and $g[m(T)](x)$ in system \eqref{MFG} and algorithm \textbf{SPI1}, by the local coupling ${\tilde f}(m)$ and $g(T,x)$ (c.f. \cref{Existence} and \cref{Bernstein}). \end{remark} 

\subsection{Algorithm \textbf{SPI2}}\label{Algo2}
{Initialize a vector field   $q^{(0)}$, $q^{(0)}\in \mathcal C^{\alpha/2,\alpha}(Q;\R^d)$ and $\| q^{(0)}\|_{L^\infty(Q;\R^d)}\leq R$}.  Iterate for each $n\geq 0$:
\begin{itemize}
	\item[\textbf{(i)}] \textbf{Generate the distribution from the current policy}. Solve
	\begin{equation}\label{m update 2}
		\left\{\begin{split}
		\qquad &	\partial_t m^{(n)}-\sigma \Delta m^{(n)}-\text{div} (m^{(n)} q^{(n)})=0,\qquad &\text{ in }Q,\\
			&m^{(n)}(0,x)=m_0(x)&\text{ in }\T^d.
		\end{split}\right.
	\end{equation}	
	\item[\textbf{(ii)}] \textbf{Smoothing}. $\bar m^{(0)}=m^{(0)}$, $\bar w^{(0)}=w^{(0)}$ and $\forall n\geq 1$, $w^{(n)}=q^{(n)}m^{(n)}$, 
	\begin{equation}\label{m bar update 2} 
	\big(\bar{m}^{(n)},\bar{w}^{(n)}\big)=(1-\frac{2}{n+1})\big(\bar{m}^{(n-1)},\bar{w}^{(n-1)}\big)+\frac{2}{n+1}\big(m^{(n)},w^{(n)}\big). 
	\end{equation}
	
	\item[\textbf{(iii)}] Let $\hat{q}^{(0)}=q^{(0)}$ and $\forall n\geq 1$: $\hat{q}^{(n)}=\frac{\bar{w}^{(n)}}{\bar{m}^{(n)}}$.
		\item[\textbf{(iv)}] Solve
	\begin{equation}\label{u update 2}
		\left\{\begin{split}
		\qquad &	-\partial_t u^{(n)}- \sigma \Delta u^{(n)}+\hat{q}^{(n)} \cdot \nabla u^{(n)}-L(x,\hat{q}^{(n)})=f[\bar{m}^{(n)}](x)&\text{ in }Q,\\
		&	u^{(n)}(T,x)=g[\bar m^{(n)}(T)](x)&\text{ in }\T^d.
		\end{split}\right.
	\end{equation}
	\item[\textbf{(v)}] 
	\begin{equation}\label{q update 2}
		q^{(n+1)}(t,x)={\arg\max}_{q\in \R^d,\,\vert q\vert \leq R}\left\{q\cdot \nabla u^{(n)}(t,x)-L(x,q)\right\}\qquad\text{ in }Q.
	\end{equation}
	 Iteration terminates if $\|q^{(n+1)}-q^{(n)}\|_{L^\infty(Q;\R^d)}$ small enough, else set $n\leftarrow n+1$ and continue.
\end{itemize}
The interpretation is that, in each stage $n$, the agent first uses a policy $q^{(n)}$ to generate the probability distribution $m^{(n)}$. The agent then uses the averaging flux $\bar{w}^{(n)}$ and distribution $\bar{m}^{(n)}$ to find $\hat{q}^{(n)}$ (the policy which would have generated $\bar{m}^{(n)}$) and takes it as the policy of her opponents and $\bar m^{(n)}$ as belief in the evolution of the distribution. She then evaluates $\hat{q}^{(n)}$ and performs a greedy update. The learning rate has the same interpretation as in Remark \ref{learning rate}. \par
One can use the same methodology as \cref{check} to check that $\big(u^{(N)},m^{(N)}\big)$ in \textbf{SPI2} is a solution to \eqref{MFG}, if ideally $\|q^{(N+1)}-q^{(N)}\|_{L^\infty(Q;\R^d)}=0$.\par
Here we denote again, 
$
\delta u^{(n+1)}=u^{(n+1)}-u^{(n)},\,\,\delta m^{(n+1)}=m^{(n+1)}-m^{(n)}.
$
\begin{lemma}\label{Solution u m2}
For each $n$, $n\geq 1$, there exists a unique solution
$
\big(u^{(n)},m^{(n)}\big)\in \mathcal C^{1+\alpha/2,2+\alpha}(Q)\times \mathcal C^{1+\alpha/2,2+\alpha}(Q)
$
to the system \eqref{m update 2}-\eqref{u update 2}. Moreover, $$\| \hat q^{(n)}\|_{L^\infty(Q;\R^d)}\leq R,$$
\begin{equation}
\|u^{(n)}\| _{\mathcal C^{1+\alpha/2,2+\alpha}(Q)}+\| m^{(n)}\| _{\mathcal C^{1+\alpha/2,2+\alpha}(Q)}+\|\bar{m}^{(n)}\| _{\mathcal C^{1+\alpha/2,2+\alpha}(Q)}\leq C,
\end{equation}
$m^{(n)}\geq 1/C$ and $\bar m^{(n)}\geq 1/C$.
\end{lemma}
\begin{proof}
From $\| q^{(n)}\|_{L^\infty(Q;\R^d)} \leq R$ we can obtain the bound on $\|m^{(n)}\|_{\mathcal C^{0}(Q)}$ from Proposition \ref{m stability} and Remark \ref{H embedding}. Then 
$
\|\bar m^{(n)}\|_{\mathcal C^{0}(Q)}, \,\,\| w^{(n)}\|_{\mathcal C^{0}(Q)}, \,\,\| \bar w^{(n)}\|_{\mathcal C^{0}(Q;\R^d)}$
are bounded independent of $n$. Since $-Rm^{(n)}(t,x)\leq w^{(n)}(t,x)\leq Rm^{(n)}(t,x)$ for all $n\geq 0$ and $(t,x)\in Q$, from \eqref{m bar update 2} we have
$-R\bar m^{(n)}(t,x)\leq \bar w^{(n)}(t,x)\leq R\bar m^{(n)}(t,x)$,
 for all $n\geq 0$ and $(t,x)\in Q$. From step (\textbf{iii}) we have $\| \hat q^{(n)}\|_{L^\infty(Q;\R^d)}\leq R$. The rest of the proof is very similar to the bootstrap argument of \cref{Solution u m1}, using \cref{linear estim} and \cref{linear estim Sobolev}, hence we omit the details.
\end{proof}
\begin{lemma}\label{C/n bound algo2}
There exist a constant $C$, such that for all $n\geq 1$, 
\begin{equation}
\begin{split}
\| \bar m^{(n+1)}-\bar m^{(n)}\|_{\mathcal C^{(1+\alpha)/2,1+\alpha}(Q)}+\| \bar w^{(n+1)}-\bar w^{(n)}\|_{\mathcal C^{\alpha/2,\alpha}(Q;\R^d)}&\\
+\| \hat q^{(n+1)}- \hat q^{(n)}\|_{\mathcal C^{\alpha/2,\alpha}(Q;\R^d)}+\| u^{(n+1)}-u^{(n)}\| _{\mathcal C^{1+\alpha/2,2+\alpha}(Q)}&\\
+\| q^{(n+1)}- q^{(n)}\|_{\mathcal C^{\alpha/2,\alpha}(Q;\R^d)}&\leq \frac{C}{n}.
\end{split}
\end{equation}
\end{lemma}
\begin{proof}
We can obtain from \eqref{m bar update 2}, step (\textbf{iii}) and \cref{Solution u m2} that for all $n\geq 1$,
$$
\| \bar m^{(n+1)}-\bar m^{(n)}\|_{\mathcal C^{(1+\alpha)/2,1+\alpha}(Q)}+\| \bar w^{(n+1)}-\bar w^{(n)}\|_{\mathcal C^{\alpha/2,\alpha}(Q;\R^d)}\leq \frac{C}{n},
$$
\begin{equation}
\begin{split}
&\|\hat q^{(n+1)}- \hat q^{(n)}\|_{\mathcal C^{\alpha/2,\alpha}(Q;\R^d)}\\
={}&\|\frac{(\bar w^{(n+1)}-\bar w^{(n)})m^{(n)}-\bar w^{(n)}(\bar m^{(n+1)}-\bar m^{(n)})}{\bar m^{(n+1)}\bar m^{(n)}}\|_{\mathcal C^{\alpha/2,\alpha}(Q;\R^d)}\leq \frac{C}{n}.
\end{split}
\end{equation}
We omit the details for the rest of the proof.
\end{proof}

\begin{theorem}\label{Maintheorem2} Under the assumptions (A1), (A2), (A3) and with $R$ sufficiently large, the family $\{ (u^n , m^n) \}$, $n \in \N$, in the algorithm \textbf{SPI2} is uniformly continuous and any cluster point is a solution to the  second order MFG \eqref{MFG}. If, in addition, the monotonicity condition (A4) holds, then the whole sequence converges to the unique solution of  \eqref{MFG}. 
\end{theorem}
\begin{proof}
From the Bernstein estimate (\cref{Bernstein}) there exists a sufficiently large $R$ such that $\|H_p(\cdot,\nabla u)\|_{L^\infty(Q;\R^d)}\leq R$.\par
 We recall that $mL(x,w/m)$ is jointly convex in the $(w,m)$ variables. Hence
\begin{equation}\label{mL convexity}
 \bar{m} ^{(n+1)}L(x,\frac{\bar w^{(n+1)}}{\bar{m} ^{(n+1)}})\leq \frac{n}{n+2} \bar{m} ^{(n)}L(x,\frac{\bar w^{(n)}}{\bar{m} ^{(n)}})+\frac{2}{n+2} m ^{(n+1)}L(x, \frac{ w^{(n+1)}}{m ^{(n+1)}}).
\end{equation}
 Moreover, for $\hat{q}^{(n+1)}$ and $\bar{w} ^{(n+1)}$ defined as in step \textbf{(iii)},
 \begin{equation}\label{mqDu}
 \begin{split}
&\bar{m} ^{(n+1)} \hat{q}^{(n+1)} \cdot \nabla u^{(n+1)}\\
= {} &\bar{w} ^{(n+1)}\cdot \nabla u^{(n)}+\bar{w} ^{(n+1)}\cdot \nabla \delta u^{(n+1)}\\
= {} &\frac{n}{n+2}\bar{w} ^{(n)}\cdot \nabla u^{(n)}+\frac{2}{n+2} m ^{(n+1)}q ^{(n+1)}\cdot \nabla u^{(n)}+\bar{w} ^{(n+1)}\cdot \nabla \delta u^{(n+1)}.
\end{split}
\end{equation}
We can then obtain, using \eqref{q update 2}, \eqref{mL convexity} and \eqref{mqDu},
 \begin{equation*}
 \begin{split}
 &\frac{2}{n+2} m ^{(n+1)} H(x,\nabla u^{(n)})\\
= {}&\frac{2}{n+2} m ^{(n+1)}q ^{(n+1)}\cdot \nabla u^{(n)}-\frac{2}{n+2} m ^{(n+1)}L(x, \frac{ w^{(n+1)}}{m ^{(n+1)}})\\
\leq  {} &\bar{m} ^{(n+1)} \hat{q}^{(n+1)} \cdot \nabla u^{(n+1)}-\frac{n}{n+2}\bar{w} ^{(n)}\cdot \nabla u^{(n)}-\bar{w} ^{(n+1)}\cdot \nabla \delta u^{(n+1)}\\
 {} &+\frac{n}{n+2} \bar{m} ^{(n)}L(x,\frac{\bar w^{(n)}}{\bar{m} ^{(n)}})-\bar{m} ^{(n+1)}L(x,\frac{\bar w^{(n+1)}}{\bar{m} ^{(n+1)}})\\
 \leq  {} & \bar{m} ^{(n+1)} \hat{q}^{(n+1)} \cdot \nabla u^{(n+1)}+\frac{2}{n+2} m ^{(n+1)}\hat q ^{(n)}\cdot \nabla u^{(n)}
 -\bar{m} ^{(n+1)}\hat{q}^{(n)}\cdot \nabla u^{(n)}\\
 {}&-\bar{w} ^{(n+1)}\cdot \nabla \delta u^{(n+1)}\\
 {}&+\bar{m} ^{(n+1)}L(x,\hat q ^{(n)})-\frac{2}{n+2} m ^{(n+1)}L(x, \hat q ^{(n)})-\bar{m} ^{(n+1)}L(x,\hat q ^{(n+1)}),
\end{split}
\end{equation*}
hence
\begin{equation}\label{H-q-L}
 \begin{split}
 &\frac{2}{n+2} m ^{(n+1)} H(x,\nabla u^{(n)})-\frac{2}{n+2} m ^{(n+1)}\big(\hat q ^{(n)}\cdot \nabla u^{(n)}- L(x, \hat q ^{(n)})\big)\\
\leq  {} & \bar{m} ^{(n+1)}\big(\hat q ^{(n+1)}\cdot \nabla u^{(n+1)}-L(x,\hat q ^{(n+1)})-\hat q ^{(n)}\cdot \nabla u^{(n)}+L(x,\hat q ^{(n)})\big)\\
&-\bar{w} ^{(n+1)}\cdot \nabla \delta u^{(n+1)}\\
 \leq  {} & \bar{m} ^{(n+1)}\partial_t \delta u^{(n+1)}+\sigma \bar{m} ^{(n+1)}\Delta \delta u^{(n+1)}-\bar{w} ^{(n+1)} \cdot \nabla \delta u^{(n+1)}\\
&+\bar{m} ^{(n+1)}\big(f[\bar{m}^{(n+1)}](x)-f[\bar{m}^{(n)}](x)\big).\\
\end{split}
\end{equation}
From Cauchy$-$Schwartz inequality and \cref{H-L}, we have
\begin{equation*}
\begin{split}
& \frac{1}{\bar{C}n} m ^{(n+1)}|H_p(x,\nabla u^{(n)})-\hat{q}^{(n)}|^2\\
\leq  {} & \frac{2}{2\bar{C}(n+2)} m ^{(n+1)}|H_p(x,\nabla u^{(n)})-\hat{q}^{(n)}|^2+\frac{C}{n^2}\\
\leq  {} &\frac{2}{n+2} m ^{(n+1)} H(x,\nabla u^{(n)})-\frac{2}{n+2} m ^{(n+1)}\big(\hat{q}^{(n)} \cdot \nabla u^{(n)}-L(x,\hat{q}^{(n)})\big)+\frac{C}{n^2},
\end{split}
\end{equation*}
then we obtain by using \eqref{H-q-L}: 
\begin{equation}\label{Algo2 integrate}
\begin{split}
& \frac{1}{\bar{C}n} m ^{(n+1)}|H_p(x,\nabla u^{(n)})-\hat{q}^{(n)}|^2\\
\leq  {} & \bar{m} ^{(n+1)}\partial_t \delta u^{(n+1)}+\sigma \bar{m} ^{(n+1)}\Delta \delta u^{(n+1)}-\bar{w} ^{(n+1)} \cdot \nabla \delta u^{(n+1)}\\
 {} &+\bar{m} ^{(n+1)}f[\bar{m}^{(n+1)}](x)-\bar{m} ^{(n)}f[\bar{m}^{(n)}](x)+f[\bar{m}^{(n)}](x)(\bar{m} ^{(n)}-\bar{m} ^{(n+1)})+\frac{C}{n^2}.
\end{split}
\end{equation}
Integrating both sides of \eqref{Algo2 integrate} on $Q$. From integration by parts, \eqref{F Lip} and \eqref{G Lip},
\begin{equation}
\begin{split}
&\frac{1}{\bar{C}}\int_Q \frac{1}{n}m ^{(n+1)}|H_p(x,\nabla u^{(n)})-\hat{q}^{(n)}|^2dxdt\\
\leq  {} &\int_0^T\left(F[\bar{m}^{(n)}(t)]-F[\bar{m}^{(n+1)}(t)]\right)dt+G[\bar{m}^{(n)}(T)]-G[\bar{m}^{(n+1)}(T)]+\frac{C}{n^2}\\
&+\int_Q\left(\bar{m} ^{(n+1)}f[\bar{m}^{(n+1)}](x)-\bar{m} ^{(n)}f[\bar{m}^{(n)}](x)\right)dxdt-\int_{\T^d}m_0(x)\delta u^{(n+1)}(0,x)dx\\
&+\int_{\T^d}\left(\bar{m}^{(n+1)}(T)g[\bar m^{(n+1)}(T)](x)-\bar{m}^{(n)}(T)g[\bar m^{(n)}(T)](x)\right)dx.
\end{split}
\end{equation}
Let
\begin{equation}\label{a_n 2}
a_n=\int_Q m ^{(n+1)}|H_p(x,\nabla u^{(n)})-\hat{q}^{(n)}|^2dxdt,
\end{equation}
by telescoping we can see that $\sum_{n=1}^{\infty} a_n / n < +\infty$. One can obtain from Lemma \ref{Solution u m2} and Lemma \ref{C/n bound algo2} $\abs a_n - a_{n+1} \abs < C/n$. Then we obtain $a_n\rightarrow 0$ from Lemma \ref{a_n} and $\{H_p(x,\nabla u^{(n)})- \hat q^{(n)}\}$ converges to $0$ uniformly on $Q$.\par
From \cref{Solution u m2}, $\{(\bar m^{(n)},\bar w^{(n)})\}$ is pre-compact for uniform convergence.
We can extract a uniformly convergent subsequence $\{(\bar m^{(n_i)},\bar w^{(n_i)})\}$ with the limit denoted by $\{(\bar m,\bar w)\}$, and hence the subsequence $\{\hat q^{(n_i)}\}$ converges uniformly to $\frac{\bar w}{\bar m}$. We use Proposition \ref{linear estim Sobolev}, ($A3$), Proposition \ref{m stability} and Remark \ref{H embedding}, pass to the limit in \eqref{m update 2} and \eqref{u update 2} to obtain
\begin{align*}
&-\partial_t u - \sigma \Delta u  +\frac{\bar w}{\bar m}\cdot \nabla u-L(x,\frac{\bar w}{\bar m}) =f[\bar m](x),\,\,u(T,x)=g[\bar m(T)](x),\\
&\partial_t \bar m- \sigma \Delta \bar m  - \text{div} ( \bar w) =0,\,\,\bar m(0,x)=m_0(x),
\end{align*}
where $(u,\bar m)\in \mathcal C^{1+\alpha/2,2+\alpha}(Q)\times \mathcal C^{1+\alpha/2,2+\alpha}(Q)$ denotes the limit of $\{(u^{(n_i)},\bar m^{(n_i)})\}$. By using ($A1$), we can obtain that $\{H_p(x,\nabla u^{(n_i)})\}$ converges to $H_p(x,\nabla u)$ unifromly. From Proposition \ref{m stability} and Remark \ref{H embedding}, $\{m^{(n_i)}\}$ converges uniformly to $m$, which is the solution to:
$$\partial_t m- \sigma \Delta  m  - \text{div} (H_p(x,\nabla u)m) =0,\,\,m(0,x)=m_0(x).$$
Since  $\{H_p(x,\nabla u^{(n_i)})- \hat q^{(n_i)}\}$ converges to $0$ uniformly, we can obtain in the limit $H_p(x,\nabla u)=\frac{\bar w}{\bar m}$ and $m=\bar m$, hence $(u,m)$ is a classical solution to system \eqref{MFG}.\par
If $(A4)$ holds, then the system \eqref{MFG} has a unique solution, so that the whole sequence $\{(u^{(n)},m^{(n)})\}$ has a unique cluster point $(u,m)$ and thus converges uniformly to $(u,m)$. 
 \end{proof}

   \subsection{Convergence to stable solutions}\label{Stable solution}
  We show \textbf{SPI2} converges to a stable solution $(u,m)$, locally in the sense that if the initial guess $q^{(0)}$ is sufficiently close to the optimal control $H_p(x,\nabla u)$. The proof for algorithm  \textbf{SPI1} is similar hence omitted. 
  \begin{theorem}
 Let $(u,m)$ be a stable solution to the MFG system \eqref{MFG}. Then there exists a $\delta>0$ such that, if $\|q^{(0)}- q\|_{L^\infty(Q;\R^d)}\leq \delta$, where $q=H_p(x,\nabla u)$, the sequence $\{(u^{(n)},m^{(n)})\}$ defined in algorithm \textbf{SPI2} converges to $(u,m)$ in $\mathcal C^{0,1}(Q)\times \mathcal C^0(Q)$ as $n\rightarrow \infty$.
 \end{theorem}
 \begin{proof}  
 Let $v=u^{(n+1)}-u$, $\rho=m^{(n+1)}-m$, $(v,\rho)$ be the solution to system \eqref{Linerized},
\begin{equation*}
\begin{split}
a(t,x):={}&-H(x,\nabla u^{(n+1)})+H(x,\nabla u)+H_p(x,\nabla u)(\nabla u^{(n+1)}-\nabla u)\\
{}&-\hat q^{(n+1)}\cdot \nabla u^{(n+1)}+L(x,\hat q^{(n+1)})+H(x,\nabla u^{(n+1)})\\
&+f[\bar{m}^{(n+1)}(t)](x)-f[m(t)](x)-\frac{\delta f}{\delta m}[m(t)](x)(m^{(n+1)}-m),
\end{split}
\end{equation*}
\begin{equation*}
\begin{split}
b(t,x):=&m^{(n+1)}H_p(x,\nabla u^{(n)})-mH_p(x,\nabla u)-(m^{(n+1)}-m)H_p(x,\nabla u)\\
&-mH_{pp}(x,\nabla u)(\nabla u^{(n)}-\nabla u) +mH_{pp}(x,\nabla u)(\nabla u^{(n)}-\nabla u^{(n+1)}),
\end{split}
\end{equation*}
$$
c(x):=g[\bar{m}^{(n+1)}(T)](x)-g[m(T)](x)-\frac{\delta g}{\delta m}[m(T)](x)\big(m^{(n+1)}(T)-m(T)\big).
$$
Recall $\hat q^{(n)}$ and $\bar w^{(n)}$ are bounded in $\mathcal C^0(Q;\R^d)$ for $n\geq 1$. From $(A2)$, \cref{Solution u m2}, \cref{C/n bound algo2} and Cauchy Schwartz inequality we can obtain:
\begin{equation*}
\begin{split}
&\|H(\cdot,\nabla u^{(n+1)})-H(\cdot,\nabla u)-H_p(\cdot,\nabla u)(\nabla u^{(n+1)}-\nabla u)\|_{\mathcal C^0(Q)}\\
\leq {}&C\|\nabla u^{(n+1)}-\nabla u\|^2_{\mathcal C^0(Q)}\\
\leq {}&C\|\nabla u^{(n)}-\nabla u\|^2_{\mathcal C^0(Q)}+\frac{C}{n^2},
\end{split}
\end{equation*}
\begin{equation*}
\begin{split}
&\|m^{(n+1)}H_p(x,\nabla u^{(n)})-mH_p(x,\nabla u)-(m^{(n+1)}-m)H_p(x,\nabla u)\\
&-mH_{pp}(x,\nabla u)(\nabla u^{(n)}-\nabla u) \|_{\mathcal C^0(Q;\R^d)}\\
\leq {}&C\big(\|\nabla u^{(n)}-\nabla u)\|^2_{\mathcal C^0(Q;\R^d)}+\|m^{(n+1)}-m\|^2_{\mathcal C^0(Q)}\big),
\end{split}
\end{equation*}
$$
\|mH_{pp}(\cdot,\nabla u)(\nabla u^{(n)}-\nabla u^{(n+1)})\|_{\mathcal C^0(Q)}\leq C\|\nabla u^{(n)}-\nabla u^{(n+1)}\|_{\mathcal C^0(Q)}\leq \frac{C}{n}.
$$
It follows from \cref{H-L} and \cref{C/n bound algo2} that
\begin{equation*}
\begin{split}
&\|H(\cdot,\nabla u^{(n+1)})+L(\cdot,\hat q^{(n+1)})-\hat q^{(n+1)}\cdot \nabla u^{(n+1)}\|_{\mathcal C^0(Q)}\\
\leq {}&C\|\hat q^{(n+1)}-H_p(\cdot,\nabla u^{(n+1)})\|_{\mathcal C^0(Q;\R^d)}\\
\leq {}&C\|\hat q^{(n)}-H_p(\cdot,\nabla u^{(n)})\|_{\mathcal C^0(Q;\R^d)}+C\|\hat q^{(n+1)}-\hat q^{(n)}\|_{\mathcal C^0(Q;\R^d)}\\
&+C\|H_p(\cdot,\nabla u^{(n+1)})-H_p(\cdot,\nabla u^{(n)})\|_{\mathcal C^0(Q;\R^d)}\\
\leq {}&C\|\hat q^{(n)}-H_p(\cdot,\nabla u^{(n)})\|_{\mathcal C^0(Q)}+\frac{C}{n}.
\end{split}
\end{equation*}

Moreover, for each $t\in [0,T]$, $\bar{m}^{(n)}(t),m^{(n)}(t)\in \mes(\T^d)\cap \mathcal C^0(\T^d)$. From \cref{F f Lip},
$$
\|f(\cdot,\bar{m}^{(n+1)})-f[m^{(n+1)}](\cdot)\|_{\mathcal C^0(Q)}\leq C\|\bar{m}^{(n+1)}-m^{(n+1)}\|_{\mathcal C^0(Q)},
$$
$$
\|f[m^{(n+1)}](\cdot)-f[m](\cdot)-\frac{\delta f}{\delta m}[m](\cdot)(m^{(n+1)}-m)\|_{\mathcal C^0(Q)}\leq C\|m^{(n+1)}-m\|^2_{\mathcal C^0(Q)},
$$
\begin{equation*}
\begin{split}
&\|g[\bar{m}^{(n+1)}(T)](\cdot)-g[m(T)](\cdot)-\frac{\delta g}{\delta m}[m(T)](\cdot)(m^{(n+1)}(T)-m(T))\|_{\mathcal C^0(\T^d)}\\
\leq{}&C\|\bar{m}^{(n+1)}(T)-m^{(n+1)}(T)\|_{\mathcal C^0(\T^d)}+C\|m^{(n+1)}(T)-m(T)\|^2_{\mathcal C^0(\T^d)}\\
\leq{}&C\|\bar{m}^{(n+1)}-m^{(n+1)}\|_{\mathcal C^0(Q)}+C\|m^{(n+1)}-m\|^2_{\mathcal C^0(Q)}.
\end{split}
\end{equation*}
From Proposition \ref{m stability} and Lemma \ref{Solution u m2}:
$$
\|m^{(n+1)}-m\|^2_{\mathcal C^0(Q)}\leq C\|H_p(\cdot,\nabla u^{(n)})-H_p(\cdot,\nabla u)\|^2_{\mathcal C^0(Q;\R^d)}\leq C\|\nabla u^{(n)}-\nabla u\|^2_{\mathcal C^0(Q;\R^d)},
$$
\begin{equation*}
\begin{split}
\|m^{(n+1)}-\bar{m}^{(n+1)}\|_{\mathcal C^0(Q)}\leq {}&C\|\hat q^{(n+1)}-H_p(\cdot,\nabla u^{(n)})\|_{\mathcal C^0(Q;\R^d)}\\
\leq {}&C\|\hat q^{(n)}-H_p(\cdot,\nabla u^{(n)})\|_{\mathcal C^0(Q;\R^d)}+\frac{C}{n}.
\end{split}
\end{equation*}
With these estimates in order, we can obtain by \cref{stable sol lemma},
\begin{equation}
\begin{split}
&\|u^{(n+1)}-u\|_{\mathcal C^{0,1}(Q)}+\|m^{(n+1)}-m\|_{\mathcal C^0(Q)}\\
\leq {}&C(\|a\|_{\mathcal C^0(Q)}+\|b\|_{\mathcal C^0(Q;\R^d)}+\|c\|_{\mathcal C^0(Q)})\\
\leq {}&C \big(\|\nabla u^{(n)}-\nabla u\|^2_{\mathcal C^0(Q;\R^d)}+\|\hat q^{(n)}-H_p(\cdot,\nabla u^{(n)})\|_{\mathcal C^0(Q;\R^d)}\big)+\frac{C}{n}.
\end{split}
\end{equation}
Hence there exists $\hat C$ depending only on the data of the problem and fixed throughout the proof,
\begin{equation}\label{2hat C bound}
\begin{split}
&\|\nabla u^{(n+1)}-\nabla u\|_{\mathcal C^0(Q;\R^d)}\\
\leq {}& \hat{C}\big(\|\nabla u^{(n)}-\nabla u\|^2_{\mathcal C^0(Q;\R^d)}+\|\hat q^{(n)}-H_p(\cdot,\nabla u^{(n)})\|_{\mathcal C^0(Q;\R^d)}\big)+\frac{\hat{C}}{n}.
\end{split}
\end{equation}
One can always find a sufficiently small positive constant $\hat \eta$ such that $\hat{C}\hat \eta<1/4$, hence
\begin{equation}\label{eta iteration1}
\hat{C}\hat \eta^2+\frac{\hat \eta}{2}<\hat \eta.
\end{equation}
In \cref{Maintheorem2} we have obtained $\|\hat q^{(n)}-H_p(x,\nabla u^{(n)})\|_{\mathcal C^0(Q;\R^d)}\rightarrow 0$, hence there exists a sufficiently large $N_0$, such that $\forall n\geq N_0$,
\begin{equation}\label{eta iteration2}
\hat{C}\|\hat q^{(n)}-H_p(\cdot,\nabla u^{(n)})\|_{\mathcal C^0(Q;\R^d)}<\frac{\hat \eta}{4}\,\,\text{and}\,\,\frac{\hat{C}}{n}<\frac{\hat \eta}{4}.
\end{equation}
For $\forall n\geq N_0$, if $\|\nabla u^{(n)}-\nabla u\|_{\mathcal C^0(Q)}\leq \hat \eta$, then \eqref{2hat C bound}, \eqref{eta iteration1} and \eqref{eta iteration2} together imply that 
$
\|\nabla u^{(n+1)}-\nabla u\|_{\mathcal C^0(Q;\R^d)}\leq \hat \eta.
$
We can obtain by induction that if
\begin{equation*}
\|\nabla u^{(N_0)}-\nabla u\|_{\mathcal C^0(Q;\R^d)}\leq \hat \eta,\,\,
\text{then}\,\|\nabla u^{(n)}-\nabla u\|_{\mathcal C^0(Q;\R^d)}\leq \hat \eta,\,\,\forall n\geq N_0+1.
\end{equation*}
 Next we can argue inductively that for a given $N_0$, $\|\nabla u^{(N_0)}-\nabla u\|_{\mathcal C^0(Q;\R^d)}\leq \hat \eta$ holds with a sufficiently close $q^{(0)}$. For each $n$, $0\leq n<N_0+1$, suppose there exists an $\eta_n>0$ such that
$
\|\hat q^{(n)}-H_p(\cdot,\nabla u)\|_{L^\infty(Q;\R^d)}\leq \hat \eta_n,
$
then again by systematically using $(A2)$, Lemma \ref{Solution u m2}, Proposition \ref{m stability} and Proposition \ref{linear estim Sobolev}, we have
$$
\|\bar{m}^{(n)}-m\|_{\mathcal C^0(Q)}\leq C\|\hat q^{(n)}-H_p(\cdot,\nabla u)\|_{L^\infty(Q;\R^d)}\leq C\hat \eta_n,
$$
\begin{equation*}
\begin{split}
\|u^{(n)}-u\|_{W^{1,2}_r(Q)}\leq {}&C\big(\|\hat q^{(n)}-H_p(\cdot,\nabla u)\|_{L^\infty(Q;\R^d)}+\|f(\cdot,\bar{m}^{(n)})-f[m](\cdot)\|_{\mathcal C^0(Q)}\\
&+\|g[\bar{m}^{(n)}(T)](\cdot)-g[m^{(n)}(T)](\cdot)\|_{\mathcal C^2(\T^d)}\big)\\
\leq {}&C\hat \eta_n,
\end{split}
\end{equation*}
\begin{equation*}
\begin{split}
\|H_p(\cdot,\nabla u^{(n)})-H_p(\cdot,\nabla u)\|_{\mathcal C^0(Q;\R^d)}&\leq C\|\nabla u^{(n)}-\nabla u\|_{\mathcal C^0(Q;\R^d)}\\
&\leq C\|u^{(n)}-u\|_{W^{1,2}_r(Q)} \leq C\hat \eta_{n},
\end{split}
\end{equation*}
\begin{equation*}
\begin{split}
\| w-\bar w^{(n)}\|_{\mathcal C^0(Q)}&=\| mH_p(\cdot,\nabla u)-\bar m^{(n)}\hat q^{(n)}\|_{L^\infty(Q;\R^d)}\\
&\leq C\big(\|\hat q^{(n)}-H_p(\cdot,\nabla u)\|_{L^\infty(Q;\R^d)}+\|\bar{m}^{(n)}-m\|_{\mathcal C^0(Q)}\big)\leq C\hat \eta_{n},
\end{split}
\end{equation*}
$$
\|m^{(n+1)}-m\|_{\mathcal C^0(Q)}\leq C\|H_p(\cdot,\nabla u^{(n)})-H_p(\cdot,\nabla u)\|_{\mathcal C^0(Q;\R^d)}\leq C\hat \eta_{n},
$$
\begin{equation*}
\begin{split}
&\| w- w^{(n+1)}\|_{\mathcal C^0(Q;\R^d)}\\
\leq {}&C\big( \|m^{(n+1)}-m\|_{\mathcal C^0(Q)}+ \|H_p(\cdot,\nabla u^{(n)})-H_p(\cdot,\nabla u)\|_{\mathcal C^0(Q;\R^d)} \big)\leq C\hat \eta_{n}.
\end{split}
\end{equation*}
Moreover, from \eqref{m bar update 2} and \cref{Solution u m2}, it is clear $\forall n\geq 0$,
$$
\| \bar w^{(n+1)}-\bar w^{(n)}\|_{L^\infty(Q;\R^d)}\leq \| w^{(n+1)}-\bar w^{(n)}\|_{L^\infty(Q;\R^d)},
$$
$$
 \|\bar m^{(n+1)}-\bar m^{(n)}\|_{\mathcal C^0(Q)}\leq \|m^{(n+1)}-\bar m^{(n)}\|_{\mathcal C^0(Q)}.
$$
We obtain from previous estimates and \cref{Solution u m2}, there exists $C_\eta>0$ such that
\begin{equation*}
\begin{split}
&\|\hat q^{(n+1)}- H_p(\cdot,\nabla u)\|_{\mathcal C^0(Q;\R^d)}\\
\leq {}&\|\hat q^{(n)}- H_p(\cdot,\nabla u)\|_{L^\infty(Q;\R^d)}+\|\hat q^{(n+1)}-\hat q^{(n)}\|_{L^\infty(Q;\R^d)}\\
\leq {}&\hat \eta_n+\|\frac{(\bar w^{(n+1)}-\bar w^{(n)})\bar m^{(n)}-\bar w^{(n)}(\bar m^{(n+1)}-\bar m^{(n)})}{\bar m^{(n+1)}\bar m^{(n)}}\|_{L^\infty(Q;\R^d)}\\
\leq {}& \hat \eta_n+C\big(\| w^{(n+1)}-\bar w^{(n)}\|_{L^\infty(Q;\R^d)}+\|m^{(n+1)}-\bar m^{(n)}\|_{\mathcal C^0(Q)}\big)\\
\leq {}&\hat \eta_n+C\big(\| w-\bar w^{(n)}\|_{L^\infty(Q;\R^d)}+\| w- w^{(n+1)}\|_{\mathcal C^0(Q;\R^d)}+\|m^{(n+1)}-m\|_{\mathcal C^0(Q)}\\
&+\|\bar m^{(n)}-m\|_{\mathcal C^0(Q)}\big)\\
\leq {}&C_\eta \hat \eta_n.
\end{split}
\end{equation*} 
Hence, if $\|q^{(0)}- q\|_{L^\infty(Q;\R^d)}\leq C_\eta^{-N_0}\eta$ then $\|\nabla u^{(N_0)}-\nabla u\|_{\mathcal C^0(Q;\R^d)}\leq \hat \eta$. \par
Finally, we consider again $n\geq N_0+1$, hence $\|\nabla u^{(n)}-\nabla u\|_{\mathcal C^0(Q;\R^d)}<\hat \eta$. From Cauchy Schwartz inequality
$
\|\nabla u^{(n)}-\nabla u\|^2_{\mathcal C^0(Q;\R^d)}\leq 2\|\nabla u^{(n+1)}-\nabla u\|^2_{\mathcal C^0(Q;\R^d)}+\frac{C}{n^2},
$
together with \eqref{2hat C bound} we have 
\begin{equation}\label{un+1-u^2}
\begin{split}
&\|\nabla u^{(n+1)}-\nabla u\|_{\mathcal C^0(Q;\R^d)}\\
\leq {}& 2\hat{C}\|\nabla u^{(n+1)}-\nabla u\|^2_{\mathcal C^0(Q;\R^d)}+\frac{C}{n^2}+\hat{C}\|\hat q^{(n)}-H_p(\cdot,\nabla u^{(n)})\|_{\mathcal C^0(Q;\R^d)}+\frac{\hat{C}}{n}.
\end{split}
\end{equation}
Since $1-2\hat C\hat \eta\geq 1/2$ then $2\hat{C}\|\nabla u^{(n)}-\nabla u\|^2_{\mathcal C^0(Q;\R^d)}\leq \frac{1}{2}\|\nabla u^{(n)}-\nabla u\|_{\mathcal C^0(Q;\R^d)}$,
 $$
\|\nabla u^{(n)}-\nabla u\|_{\mathcal C^0(Q;\R^d)}\leq  2\hat{C}\|\hat q^{(n-1)}-H_p(\cdot,\nabla u^{(n-1)})\|_{\mathcal C^0(Q;\R^d)}+\frac{2C}{(n-1)^2}+\frac{2\hat{C}}{n-1}.
$$
Let $n\rightarrow \infty$, then $\|\nabla u^{(n)}-\nabla u\|_{\mathcal C^0(Q;\R^d)}\rightarrow 0$. Moreover from $(A2)$, $H_p(x,\nabla u^{(n)})$ converges uniformly to $H_p(x,\nabla u)$, from Proposition \ref{m stability} and Remark \ref{H embedding} we have $\{m^{(n)}\}$ converges to $m$ in $\mathcal C^0(Q)$. From Proposition \ref{linear estim Sobolev} and \eqref{embedding}, $\{u^{(n)}\}$ convergences to $u$ in $\mathcal C^{0,1}(Q)$.
\end{proof}

\section{Numerical simulations}\label{Numerical}
\subsection{Numerical method}
We implement the \textbf{SPI} algorithms with finite difference schemes. They generalize the schemes developed in \cite{ccg,lst}. To save notations we only discuss the method on some particular one-dimensional cases, which can be generalized e.g. to dimension 2. An interesting possible extension will be to numerical analysis of mean field games of controls \cite{ak}, where agents are coupled via nonlocal interactions between their controls. Following \cref{Bernstein}, the bound $R$ for $q$ comes from theoretical analysis and is in general not trackable. In the implementation of the scheme we choose $R=10000$.\par
 \noindent \textbf{Discretization}. We discrete the problem on $Q=[0,T]\times [-1,1]$ with grid $\mathcal{G}$. $h=2/\mathsf{I}$, $i=0,...,\mathsf{I}$. $\Delta t=T/\mathsf{T}$, $\tau=0,...,\mathsf{T}$. Set $t_\tau=\tau \Delta t$ and $x_i=ih$. The values of $u$ and $m$ at $(t_\tau,x_i)$ are approximated by $U_{\tau,i}$ and $M_{\tau,i}$, respectively. Fix $U_{\mathsf{T},i}=u(T,x_i),\,\,M_{0,i}=m(0,x_i)$ and $V_i=V(x_i)$ for all $i$. Denote by $(\cdot)^+=\max\left\{\cdot,0\right\}$ and $(\cdot)^-=\min\left\{\cdot,0\right\}$ for the positive and negative part respectively of a real number. We define the following discrete operators:
\begin{align*}
	(\Delta_\sharp U_{\tau})_i &=\frac{1}{h^2}\left(U_{\tau,[i-1]}-2U_{\tau,i}+U_{\tau,[i+1]}\right)\,,
	\\
	(D_\sharp U_{\tau})_i &=\left(D_L U_{\tau,i}\,,\,D_R U_{\tau,i}\right)=\frac{1}{h}\left( U_{\tau,i}-U_{\tau,[i-1]}\,,\,U_{\tau,[i+1]}-U_{\tau,i}\right)\,,
	\end{align*}
	\begin{equation*}
	\begin{split}
	\text{div}_\sharp(M_{\tau+1}\,D_\sharp U_{\tau})_i={}&\frac{1}{h}
	\left( M_{\tau+1,[i+1]}D_LU_{\tau,[i+1]}^+ - M_{\tau+1,i} D_LU_{\tau,i}^+ \right)\\
	&+ \frac{1}{h}\left(M_{\tau+1,i} D_RU_{\tau,i}^- -M_{\tau+1,[i-1]} D_RU_{\tau,[i-1]}^-\right),
	\end{split}
\end{equation*}
	where the index operator $[\cdot]=\left\{(\cdot +I)\,mod\, I\right\}$ accounts for the periodic boundary conditions.  The two sided discrete gradient operator $(D_\sharp U_{\tau})_i$ is commonly used in construction of numerical schemes for MFGs \cite{ad,acd}. The discrete divergence operator captures the essential structural features of the MFG system \cite{ad,acd}. This motivates, as in \cite{ccg,lst}, the following two sided discretization of policy, with $\mathsf{Q}: \mathcal{G}\rightarrow \mathbb{R}^2$:
$$
	\mathsf{Q}_{\tau,i}=\left(\mathsf{Q}_{\tau,i,L}\,,\,\mathsf{Q}_{\tau,i,R}\right),\,\,\mathsf{Q}_{\tau,i,\pm}=\left(\mathsf{Q}^+_{\tau,i,L}\,,\,\mathsf{Q}^-_{\tau,i,R}\right).
$$
We denote
$
	\vert \mathsf{Q}_{\tau,i,\pm}\vert ^2=\left( \mathsf{Q}_{\tau,i,L}^+\right)^2+\left(\mathsf{Q}_{\tau,i,R}^-\right)^2,\,\,\vert \mathsf{Q}_{\tau,i,\pm}\vert _\infty=\sup \left\{\mathsf{Q}^+_{\tau,i,L}\,,\,-\mathsf{Q}^-_{\tau,i,R}\right\}
$
and the discrete divergence operator
\begin{equation}\label{discrete div}
\begin{split}
	\text{div}_\sharp(M_{\tau+1}\,\mathsf{Q}_\tau)_i
	={}&\frac{1}{h}
	\left( M_{\tau+1,[i+1]}\mathsf{Q}_{\tau,[i+1],L}^+ - M_{\tau+1,i} \mathsf{Q}_{\tau,i,L}^+ \right)\\
	&+ \frac{1}{h}\left(
	M_{\tau+1,i} \mathsf{Q}_{\tau,i,R}^- -M_{\tau+1,[i-1]} \mathsf{Q}_{\tau,[i-1],R}^-
	\right).
	\end{split}
\end{equation}
Following \cite[(12)]{ad}, we approximate the coupling $f[m(t)](x)$ and $g[M(T)](x)$ by 
\begin{equation}\label{discretize f}
\big(f_h[M_\tau]\big)_i=f[(m_\tau)_h](x_i), \,\,\big(g_h[M_\mathsf{T}]\big)_i=g[(m_\mathsf{T})_h](x_i), 
\end{equation}
where $(m_\tau)_h$ denotes the piecewise constant function taking the value $M_{\tau,i}$ in the square $|x-x_i|\leq h/2$.\par
\noindent \textbf{Numerical method}. We now define the discrete Smoothed Policy Iteration algorithm (discrete \textbf{SPI1}). Initialize with $\mathsf{Q}^{(0)}_{\tau,i}$ such that $\mathsf{Q}^{(0)}_{\tau,i}=\bar{\mathsf{Q}}^{(0)}_{\tau,i}$, $\vert \mathsf{Q}^{(0)}_{\tau,i,\pm}\vert_\infty \leq R$, for all $\tau, i$. Iterate:
\begin{itemize}
	\item[\textbf{(i)}]  Solve for $\tau=0,\dots, \mathsf{T}-1$, $i=0,...,\mathsf{I}-1$, $M^{(n)}_{0,i}=M_{0,i}$,
	\begin{equation*}\label{M update1}
	\frac{M^{(n)}_{\tau+1,i}-M^{(n)}_{\tau,i}}{\Delta t}-\sigma(\Delta_\sharp M^{(n)}_{\tau+1})_i-\text{div}_\sharp(M_{\tau+1}\,\bar{\mathsf{Q}}^{(n)}_\tau)_i=0.
	\end{equation*}
	
	\item[\textbf{(ii)}] Solve for $\tau=0,\dots, \mathsf{T}-1$, $i=0,...,\mathsf{I}-1$, $U^{(n)}_{\mathsf{T},i}=\big(g_h[M^{(n)}_{\mathsf{T}}]\big)_i$,
	\begin{equation*}\label{U update1}
\frac{U_{\tau,i}^{(n)}-U_{\tau+1,i}^{(n)}}{\Delta t}- \sigma(\Delta_\sharp U^{(n)}_{\tau})_i+\bar{\mathsf{Q}}^{(n)}_{\tau,i,\pm}\cdot (D_\sharp U^{(n)}_{\tau})_i- \frac{1}{2}\vert \bar{\mathsf{Q}}^{(n)}_{\tau,i,\pm}\vert ^2-V_{i}=\big(f_h[M^{(n)}_{\tau+1}]\big)_i.
	 \end{equation*}
	\item[\textbf{(iii)}] For  $\tau=0,\dots, \mathsf{T}-1$, $i=0,...,\mathsf{I}$,
	\begin{equation*}
	\mathsf{Q}^{(n+1)}_{\tau,i}=\left(\min \left\{R,(D_L U^{(n)}_{\tau,i})^+\right\},\,\,\max \left\{-R,(D_R U^{(n)}_{\tau,i})^-\right\}\right).
	\end{equation*}
	 If $\|\mathsf{Q}^{(n+1)}-\mathsf{Q}^{(n)}\|_{l^\infty}$ is small enough, stop. Otherwise continue.	
	\item[\textbf{(iv)}] For $\tau=0,\dots, \mathsf{T}-1$, $i=0,...,\mathsf{I}$,
	\begin{equation}
	\bar{\mathsf{Q}}_{\tau,i,\pm}^{(n+1)}=(1-\frac{2}{n+2})\bar{\mathsf{Q}}_{\tau,i,\pm}^{(n)}+\frac{2}{n+2}\mathsf{Q}_{\tau,i,\pm}^{(n)}.
	\end{equation}
	Set $n\leftarrow n+1$ and continue.
\end{itemize}
In particular, we set 
\begin{equation}\label{sup q-q}
\|\mathsf{Q}^{(n+1)}-\mathsf{Q}^{(n)}\|_{l^\infty} =\max_{\tau,i}\max \left\{\vert (\mathsf{Q}^{(n+1)}_{\tau,i,L})^+-(\mathsf{Q}^{(n)}_{\tau,i,L})^+\vert,\vert (\mathsf{Q}^{(n+1)}_{\tau,i,R})^--(\mathsf{Q}^{(n)}_{\tau,i,R})^-\vert \right\}.
\end{equation}
We now define the discrete Smoothed Policy Iteration algorithm (discrete \textbf{SPI2}). Initialize with $\mathsf{Q}^{(0)}_{\tau,i}$ such that $\vert \mathsf{Q}^{(0)}_{\tau,i,\pm}\vert \leq R$, for all $\tau, i$. Iterate:
\begin{itemize}
	\item[\textbf{(i)}]  Solve for $\tau=0,\dots, \mathsf{T}-1$, $i=0,...,\mathsf{I}-1$, $M^{(n)}_{0,i}=M_{0,i}$,
	\begin{equation*}\label{M update}
	 \frac{M^{(n)}_{\tau+1,i}-M^{(n)}_{\tau,i}}{\Delta t}-\sigma\big(\Delta_\sharp M^{(n)}_{\tau+1}\big)_i-\text{div}_\sharp(M_{\tau+1}\,\mathsf{Q}^{(n)}_\tau)_i=0.
	\end{equation*}
	\item[\textbf{(ii)}] For $\tau=0,\dots, \mathsf{T}$, $i=0,...,\mathsf{I}$, $\bar{M}_{\tau,i}^{(0)}=M_{\tau,i}^{(0)}$ and $\bar{W}_{\tau+1,i,\pm}^{(0)}=W_{\tau+1,i,\pm}^{(0)}$, $\forall n\geq 1$,
	$$
	W_{\tau+1,i,\pm}^{(n)}=M_{\tau+1,i}^{(n)}\mathsf{Q}^{(n)}_{\tau,i,\pm}=\left(M_{\tau+1,i}^{(n)}\big(\mathsf{Q}_{\tau,i,L}^{(n)}\big)^+,M_{\tau+1,i}^{(n)}\big(\mathsf{Q}_{\tau,i,R}^{(n)}\big)^-\right),
	$$
	$$
	\big(\bar{M}_{\tau,i}^{(n)},\bar{W}_{\tau+1,i,\pm}^{(n)}\big)=(1-\frac{2}{n+1})\big(\bar{M}_{\tau,i}^{(n-1)},\bar{W}_{\tau+1,i,\pm}^{(n-1)}\big)+\frac{2}{n+1}\big(M_{\tau,i}^{(n)},W_{\tau+1,i,\pm}^{(n)}\big).
	$$
	\item[\textbf{(iii)}] For $\tau=0,\dots, \mathsf{T}-1$, $i=0,...,\mathsf{I}$, $\hat{\mathsf{Q}}_{\tau,i,\pm}^{(0)}=\mathsf{Q}^{(0)}_{\tau,i,\pm}$ and $\forall n\geq 1$:
	\begin{equation*}
	\hat{\mathsf{Q}}_{\tau,i,\pm}^{(n)}=\left((\hat{\mathsf{Q}}_{\tau,i,L}^{(n)})^+,(\hat{\mathsf{Q}}_{\tau,i,R}^{(n)})^-\right)=\left(\frac{(\bar{W}_{\tau+1,i,L}^{(n)})^+}{\bar{M}_{\tau+1,i}^{(n)}},\frac{(\bar{W}_{\tau+1,i,R}^{(n)})^-}{\bar{M}_{\tau+1,i}^{(n)}}\right).
	\end{equation*}

	\item[\textbf{(iv)}]  Solve for $\tau=0,\dots, \mathsf{T}-1$, $i=0,...,\mathsf{I}-1$, $U^{(n)}_{\mathsf{T},i}=\big(g_h[\bar M^{(n)}_{\mathsf{T}}]\big)_i$,
	\begin{equation*}\label{U update}
 \frac{U_{\tau,i}^{(n)}-U_{\tau+1,i}^{(n)}}{\Delta t}- \sigma \Delta_\sharp U^{(n)}_{\tau,i}+\hat{\mathsf{Q}}^{(n)}_{\tau,i,\pm}\cdot (D_\sharp U^{(n)}_{\tau})_i- \frac{1}{2}\vert \hat{\mathsf{Q}}^{(n)}_{\tau,i,\pm}\vert ^2-V_i=\big(f_h[\bar M^{(n)}_{\tau+1}]\big)_i.
	 \end{equation*}
	\item[\textbf{(v)}] Update the policy,  $\tau=0,\dots, \mathsf{T}-1$, $i=0,...,\mathsf{I}$,
	\begin{equation*}
	\mathsf{Q}^{(n+1)}_{\tau,i}=\left(\min \left\{R,(D_L U^{(n)}_{\tau,i})^+\right\},\,\,\max \left\{-R,(D_R U^{(n)}_{\tau,i})^-\right\}\right).
	\end{equation*}
	If $\|\mathsf{Q}^{(n+1)}-\mathsf{Q}^{(n)}\|_{l^\infty}$ is small enough, stop. Otherwise set $n\leftarrow n+1$ and continue.
	 \end{itemize}
	 $\|\mathsf{Q}^{(n+1)}-\mathsf{Q}^{(n)}\|_{l^\infty}$ in discrete \textbf{SPI2} is defined the same as \eqref{sup q-q}.	 

\subsection{Numerical results}
\subsubsection{Results in $1D$}
For all the $1D$ examples we consider the domain $Q:=[0,1]\times [-1,1]$. $\Delta t=0.005$ and we set $i=0,...,\mathsf{I}$, $\mathsf{I}=200$ with fictitious points $i=-1$ and $i=\mathsf{I}+1$. $H(x,\nabla u)=\frac{1}{2}|\nabla u|^2-V(x)$, the interaction terms and their discretized forms with:
\begin{align*}
f[(m(t)](x)=\theta \int_{-1}^{1}l (x-y)m(t,y)dy,\,\,\big(f_h[M_\tau]\big)_i=\theta \sum_j hl\big((i-j)h\big)M_{\tau,j},\\
g[(m(T)](x)=\eta \int_{-1}^{1}l (x-y)m(T,y)dy,\,\,\big(g_h[M_\mathsf{T}]\big)_i=\eta \sum_j hl\big((i-j)h\big)M_{\mathsf{T},j}.
\end{align*}
We first consider test 1 and test 2 with periodic boundary conditions: for all $t\in [0,1]$, $u(t,-1)=u(t,1)$ and $m(t,-1)=m(t,1)$. We set $m_0=\frac{1}{2}(\cos(\pi x)+1)$, $\theta=1$, $\eta=0.2$, $V(x)=0$, $l(x-y)=\sin (\pi(x-y))$ discretized by $l\big((i-j)h\big)=\sin \big(\pi(i-j)h\big)$. The only difference is that $\eta=0.2$ in test 1 and $\eta=-0.5$ in test 2. For all $\tau$ we set $M_{\tau,-1}=M_{\tau,\mathsf{I}}$ and $M_{\tau,0}=M_{\tau,\mathsf{I}+1}$. $U_{\tau,i}$ is considered likewise. It is known from \cite[3.4.2, Example 4]{carmona2018} that, since $f$ and $g$ are odd functions, the monotonicity condition $(A4)$ is satisfied. With a smooth initial density $m_0$ the solution should be unique. \par
In \cref{fig:test12decay0} and \cref{fig:test12decay10x}, we show the convergence of the algorithms by plotting $\|\mathsf{Q}^{(n+1)}-\mathsf{Q}^{(n)}\|_{l^\infty}$ with different initial guesses. Different numbers of iterations are needed for test 1 and test 2 to reach a similar level of precision. The convergence performances with \textbf{SPI1} and \textbf{SPI2} are indistinguishable if we use the initial guess $q^{(0)}=0$. The guess $q^{(0)}=10x$ is not continuous at the boundary and the performances are clearly inferior. We observe differences between \textbf{SPI1} and \textbf{SPI2} in their speed of convergence, but multiple tests with other guesses indicate there is no clear advantage with either \textbf{SPI1} or \textbf{SPI2} versus the other.
\begin{figure}[h!]
	\centering
	\caption{Decay of $\|\mathsf{Q}^{(n+1)}-\mathsf{Q}^{(n)}\|_{l^\infty}$ for test 1 (left) and test 2 (right), $q^{(0)}=0$}\label{fig:test12decay0}
	\begin{tabular}{cc}
	\includegraphics[width=0.46\textwidth]{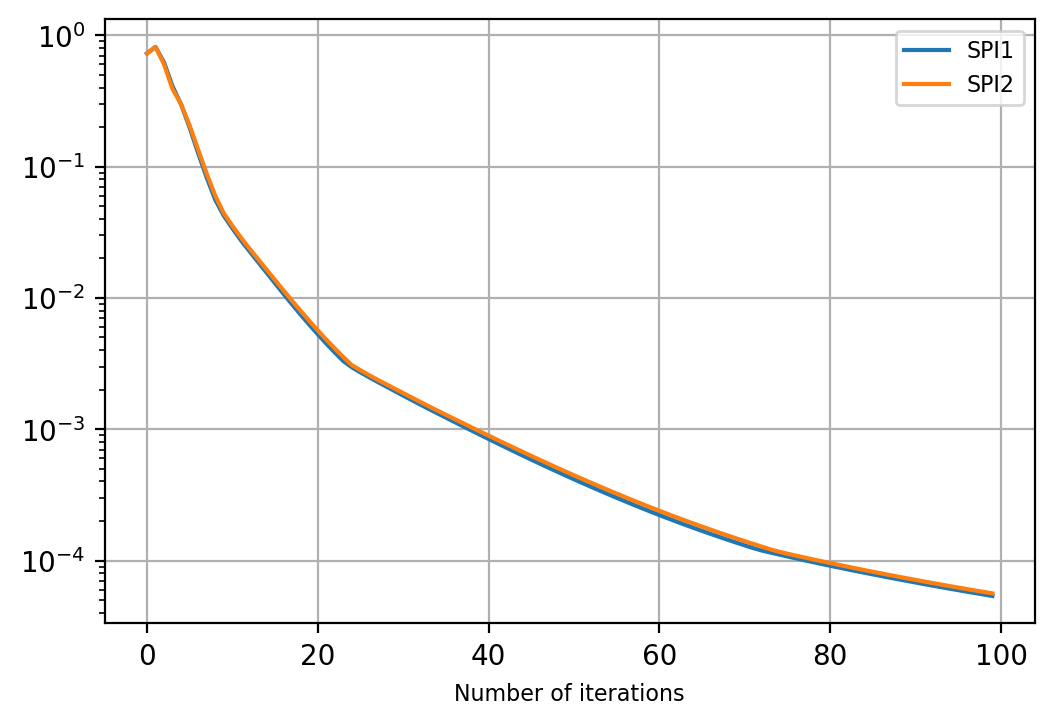} &
	\includegraphics[width=0.46\textwidth]{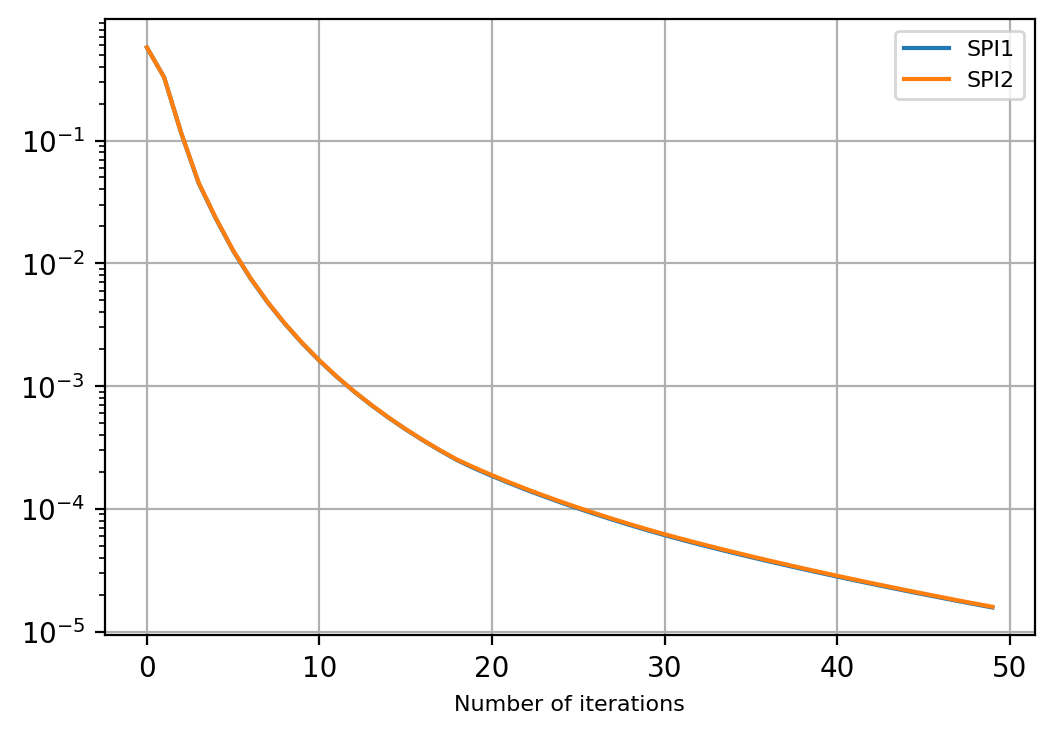} \\
	\end{tabular}
\end{figure}

\begin{figure}[h!]
	\centering
	\caption{Decay of $\|\mathsf{Q}^{(n+1)}-\mathsf{Q}^{(n)}\|_{l^\infty}$ for test 1 (left) and test 2 (right), $q^{(0)}=10x$}\label{fig:test12decay10x}
	\begin{tabular}{cc}
	\includegraphics[width=0.46\textwidth]{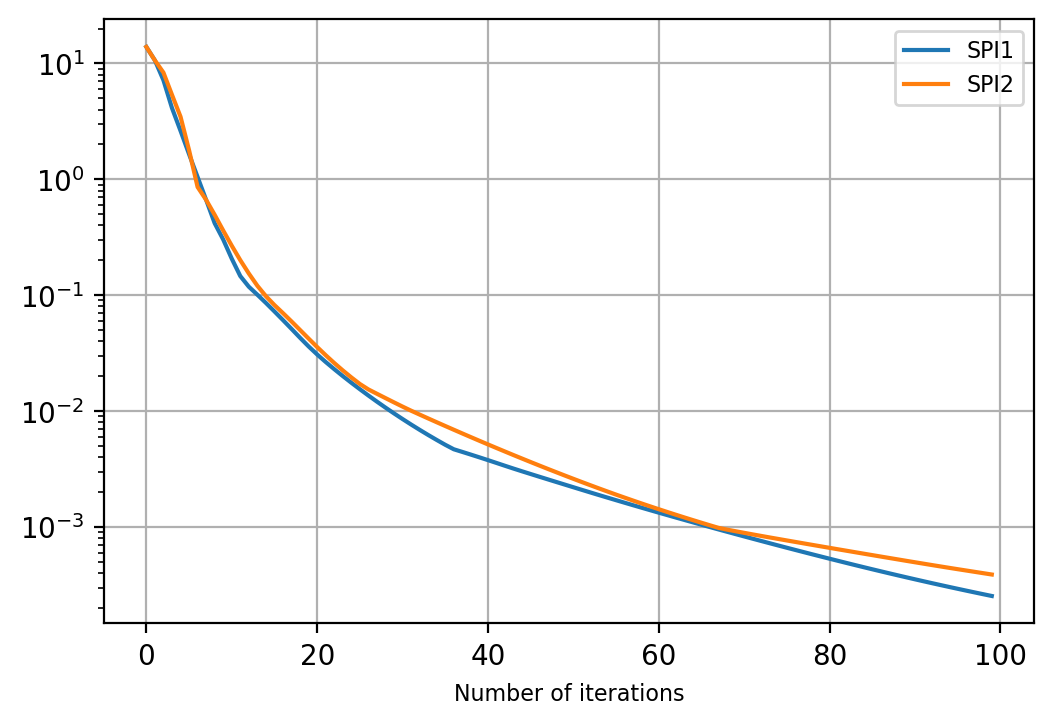} &
	\includegraphics[width=0.46\textwidth]{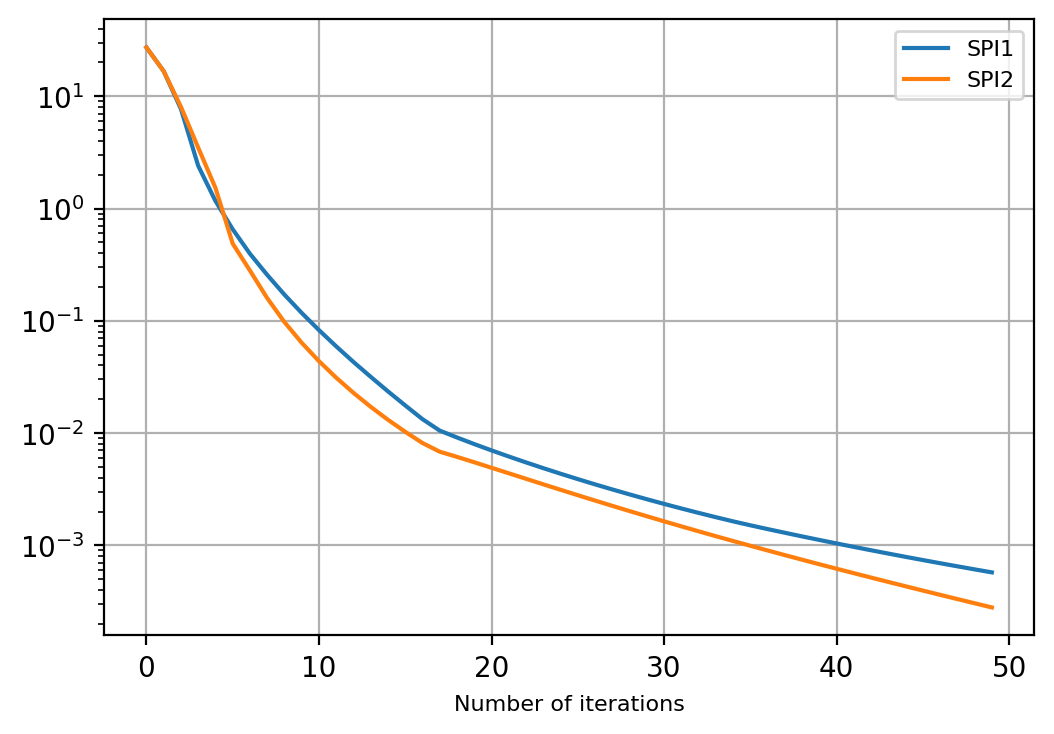} \\
	\end{tabular}
\end{figure}

\begin{figure}[h!]
	\centering
	\caption{Test 1: evolution of $m$ (left) and $u$ (right)}\label{fig:test1}
	\begin{tabular}{cc}
	\includegraphics[width=0.46\textwidth]{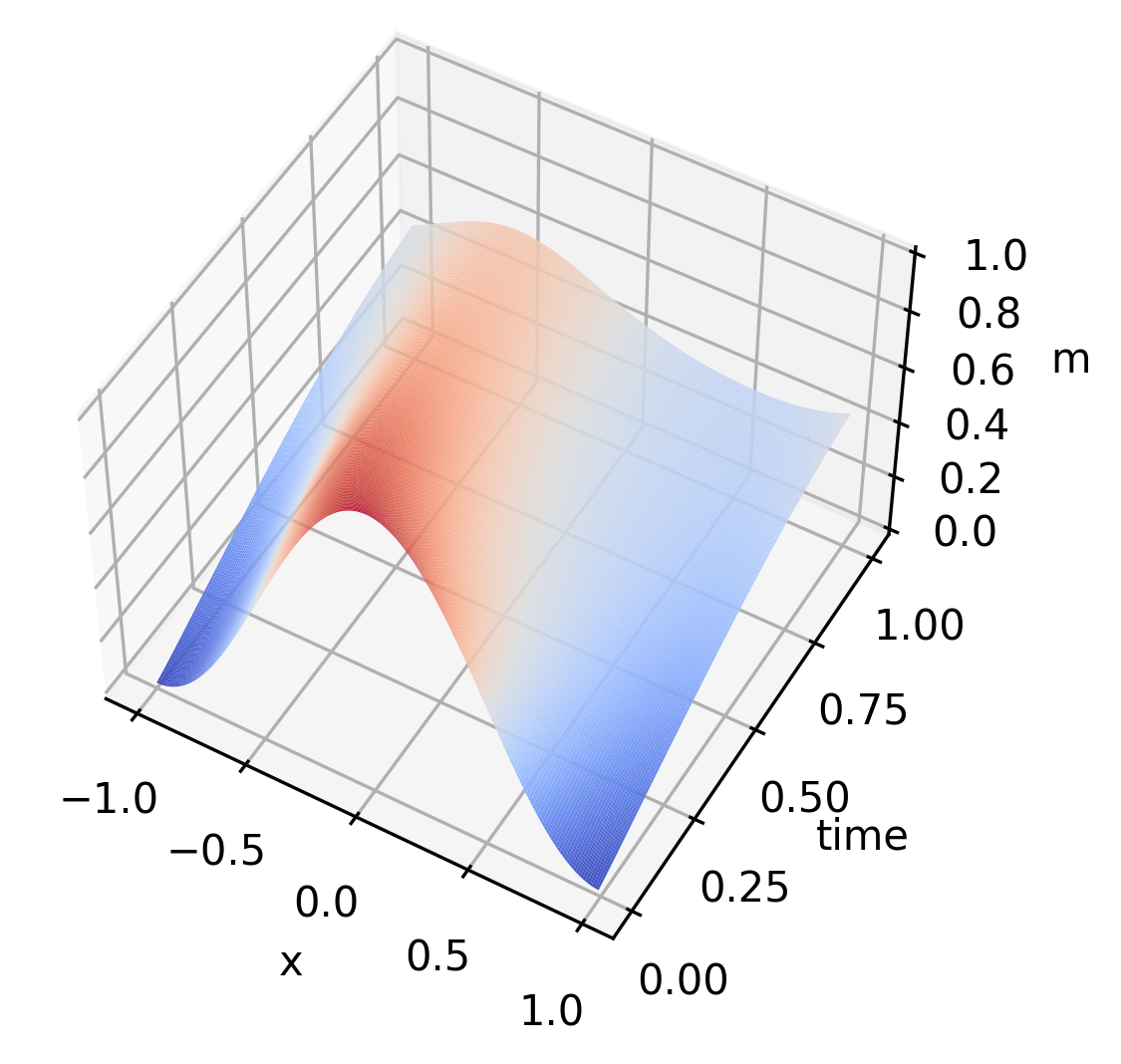} &
	\includegraphics[width=0.46\textwidth]{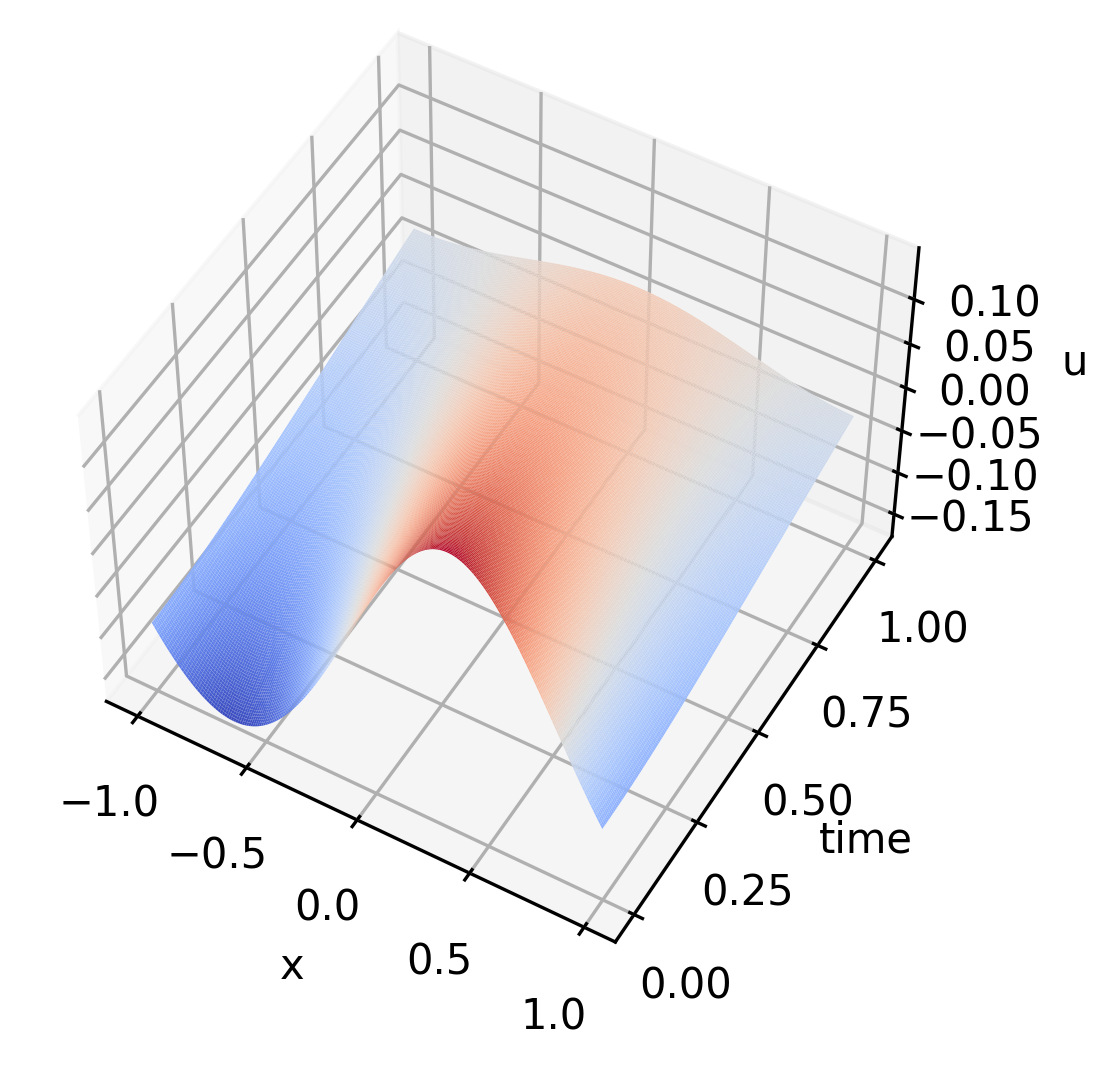} \\
	\end{tabular}
\end{figure}

\begin{figure}[h!]
	\centering
	\caption{Test 2: evolution of $m$ (left) and $u$ (right)}\label{fig:test2}
	\begin{tabular}{cc}
	\includegraphics[width=0.46\textwidth]{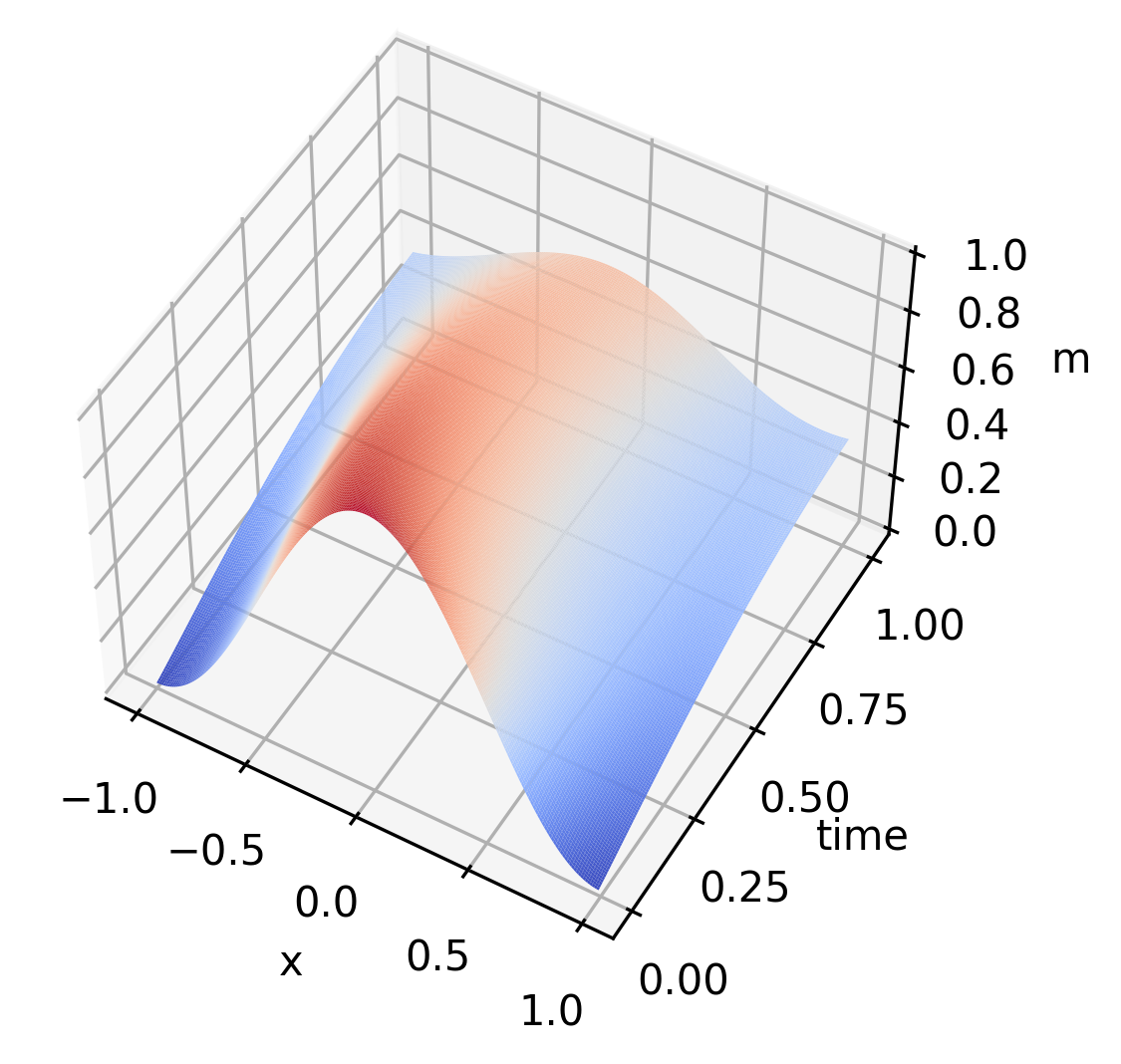} &
	\includegraphics[width=0.46\textwidth]{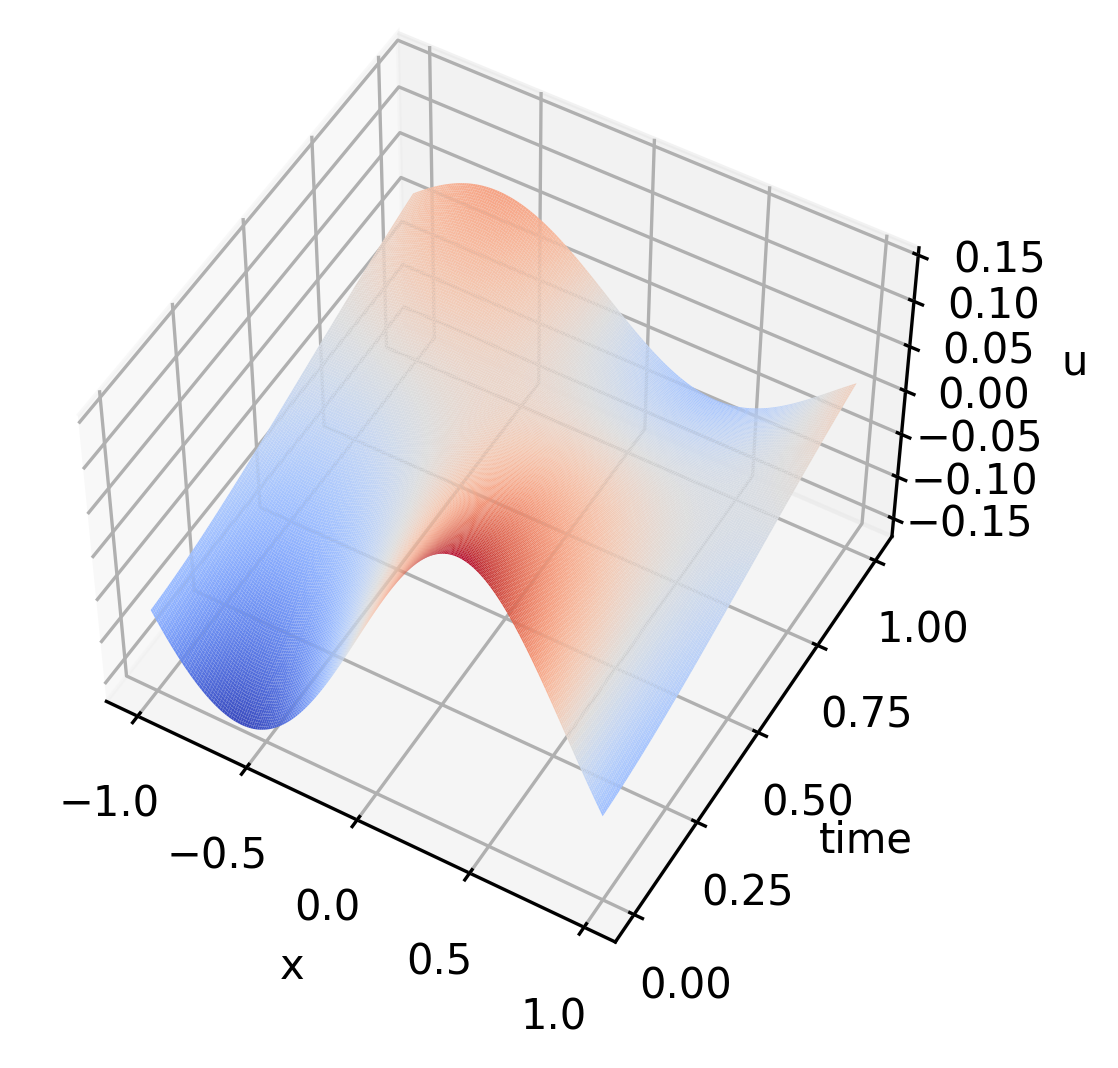} \\
	\end{tabular}
\end{figure}
Next we consider another example with Neumann boundary condition: we set $\partial_x u(t,-1)=\partial_x u(t,1)=0$ and $\partial_xm(t,-1)=\partial_x m(t,1)=0$ for all $t\in [0,1]$.\par
\begin{itemize}
\item Test 3: $\theta=1$, $\eta=0.2$, $\zeta=0.2$, $V(x)=(x+0.5)^2$, $l(x-y)=e^{-\zeta (x-y)^2}$.
\end{itemize}
We discretize $l(\cdot)$ by $l\big((i-j)h\big)=e^{-\zeta \big((i-j)h\big)^2}$. For all $\tau$ we set $M_{\tau,-1}=M_{\tau,0}$ and $M_{\tau,\mathsf{I}}=M_{\tau,\mathsf{I}+1}$. $U_{\tau,i}$ is considered likewise. The monotonicity condition $(A4)$ is satisfied for test 3, which can be shown by similar reasoning as \cite[Example 5, p. 173]{carmona2018}. Moreover, since $l(\cdot)$ is an even function, potential \eqref{Potential} can be formulated as (c.f. \cite[(6.134)-(6.135), p. 604]{carmona2018}):
$$
F[(m(t)]=\frac{1}{2}\theta \int_{-1}^{1}\int_{-1}^{1}e^{-\zeta (x-y)^2}m(t,y)m(t,x)dydx,
$$
$$
F_h[M_\tau]=\frac{1}{2}\theta \sum_i\sum_j h^2e^{-\zeta \big((i-j)h\big)^2}M_{\tau,j}M_{\tau,i},
$$
$$
G[(m(T)]=\frac{1}{2}\eta \int_{-1}^{1}\int_{-1}^{1}e^{-\zeta (x-y)^2}m(T,y)m(T,x)dydx,
$$
$$
G_h[M_\mathsf{T}]=\frac{1}{2}\eta \sum_i\sum_j h^2e^{-\zeta \big((i-j)h\big)^2}M_{\mathsf{T},j}M_{\mathsf{T},i}.
$$
We can consider a discretized version of $J_{t_0}$ (plotted in \cref{fig:test3decay}) as:
$$
\mathsf{J}_{\tau'}=\sum_{\tau=\tau'}^\mathsf{T}\Delta t\left(\sum_i h M_{\tau,i}\big(\frac{1}{2}\vert \mathsf{Q}_{\tau,i,\pm}\vert ^2+V_i\big)+F_h[M_\tau]\right)+G_h[M_\mathsf{T}].
$$
\begin{figure}[h!]
	\centering
	\caption{Test 3: Decay of $\|\mathsf{Q}^{(n+1)}-\mathsf{Q}^{(n)}\|_{l^\infty}$ (left) and $\mathsf{J}_{\tau'}$ (right)}\label{fig:test3decay}
	\begin{tabular}{cc}
	\includegraphics[width=0.46\textwidth]{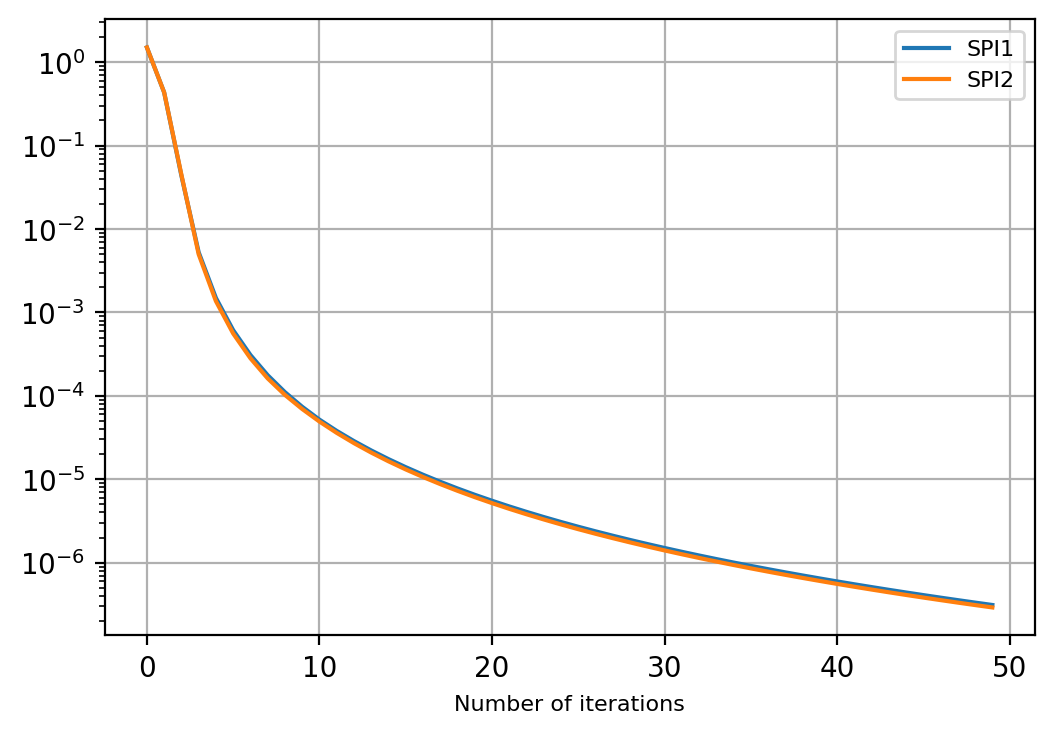} &
	\includegraphics[width=0.46\textwidth]{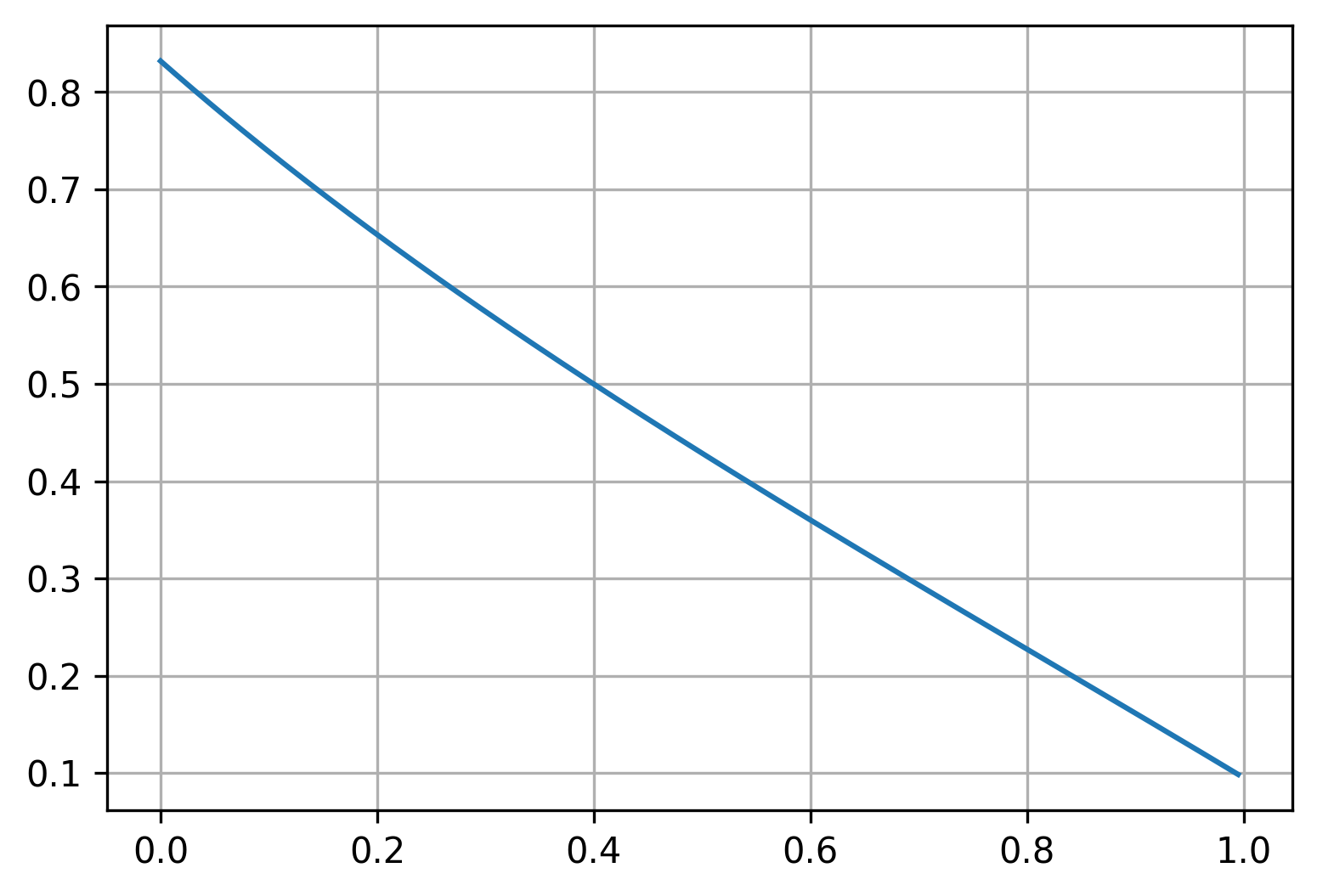} \\
	\end{tabular}
\end{figure}
\begin{figure}[h!]
	\centering
	\caption{Test 3: evolution of $m$ (left) and $u$ (right)}\label{fig:test3}
	\begin{tabular}{cc}
	\includegraphics[width=0.46\textwidth]{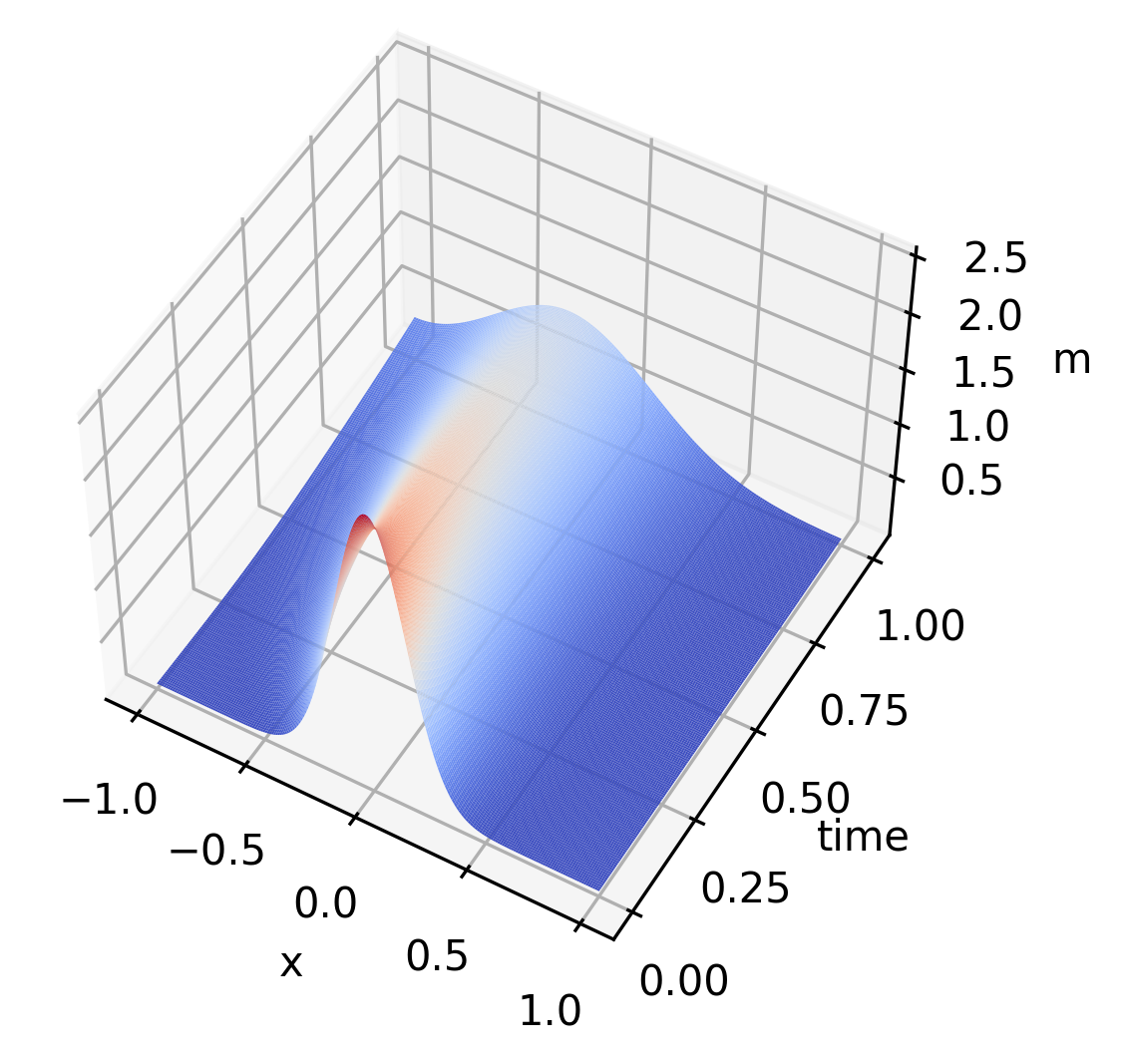} &
	\includegraphics[width=0.46\textwidth]{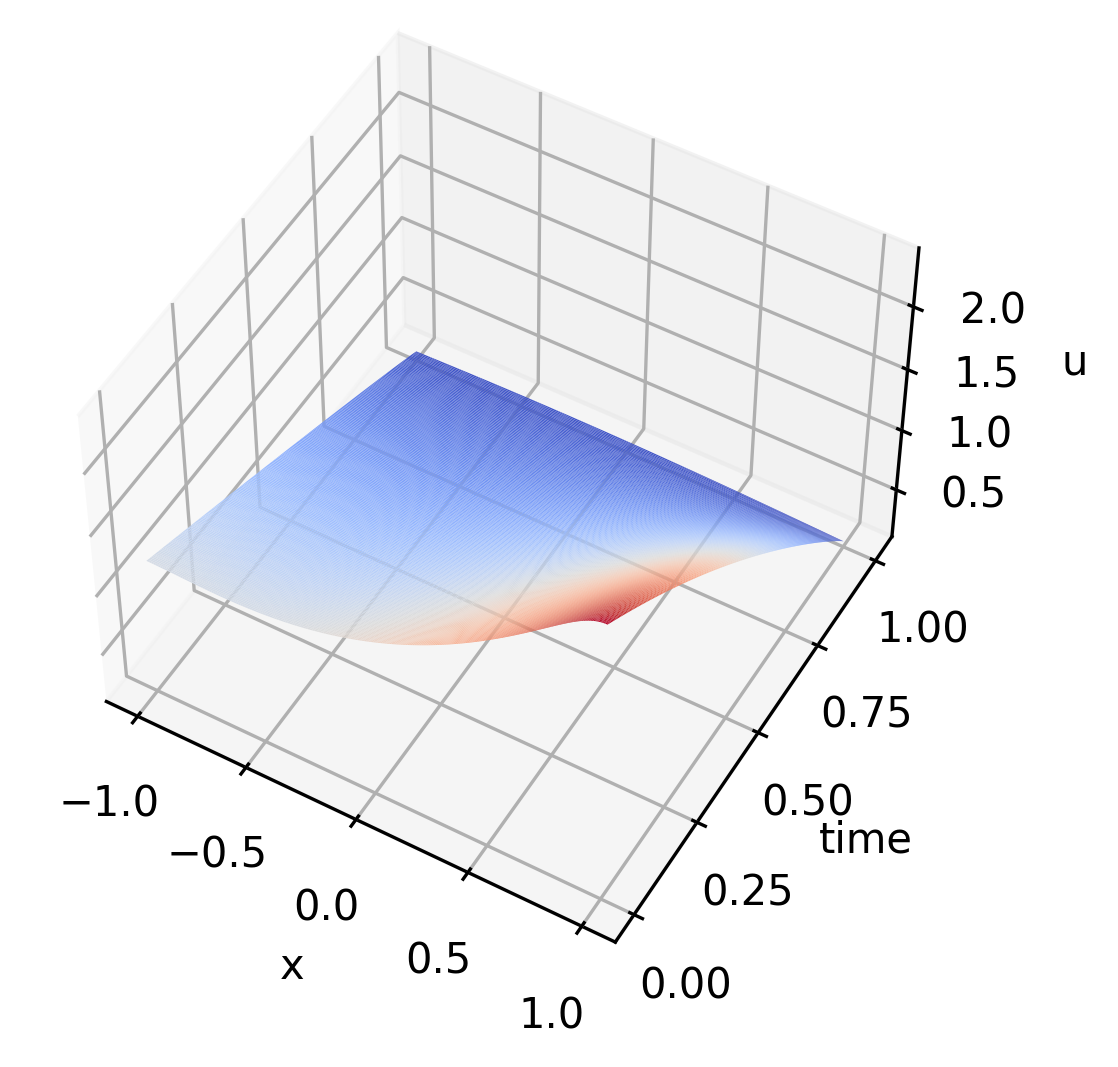} \\
	\end{tabular}
\end{figure}

\subsubsection{Results in $2D$}
We provide a test case in $2D$ with anti-monotone and local coupling. We use the space domain $(x_1,x_2)\in (0,1)^2$, $\sigma=0.25$, $T=0.5$ and Neumann boundary condition. We follow \cite[(2.16), p. 8]{ak} for numerical approximation method with Neumann boundary condition. For the results in this paragraph we have used $h=0.01$ in both space dimensions and $\Delta t=0.01$.  
$
V(x)=5\big(\cos (2\pi x_1)+\cos (2\pi x_2)\big),\,\,{\tilde{f}}(m)=-1.5m^{\frac{4}{5}},
$
$
u_T(x)=-2\big(e^{-10(x_1-0.8)^2}+e^{-10(x_2-0.8)^2}\big),
$
$$
m_0(x)=\frac{e^{-20(x_1-0.2)^2}+e^{-20(x_2-0.2)^2}}{\int_0^1\int_0^1\big(e^{-20(x'_1-0.2)^2}+e^{-20(x'_2-0.2)^2}\big)dx'_1dx'_2}.
$$
We illustrate the evolution of density and the turnpike phenomena in \cref{fig:2d}. Turnpike phenomena for MFG with anti-monotone and local coupling has been discussed in Cirant and Porreta \cite{Cirant2021}.\\\\

\begin{figure}[h!]
   \caption{Density evolution in the 2D test case}\label{fig:2d}
\begin{tabular}{cc}
\includegraphics[width=.46\textwidth]{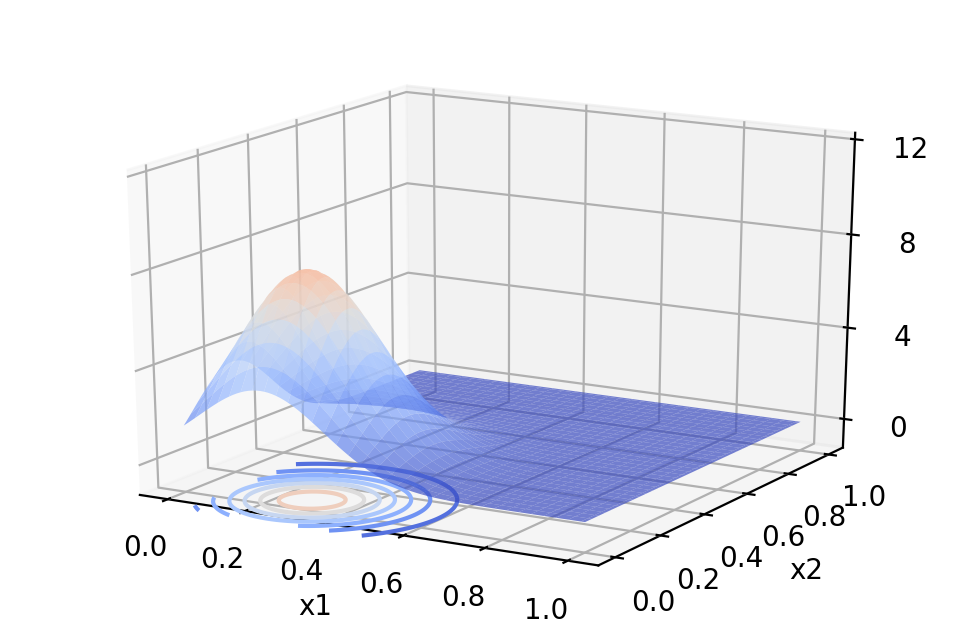} &
\includegraphics[width=.46\textwidth]{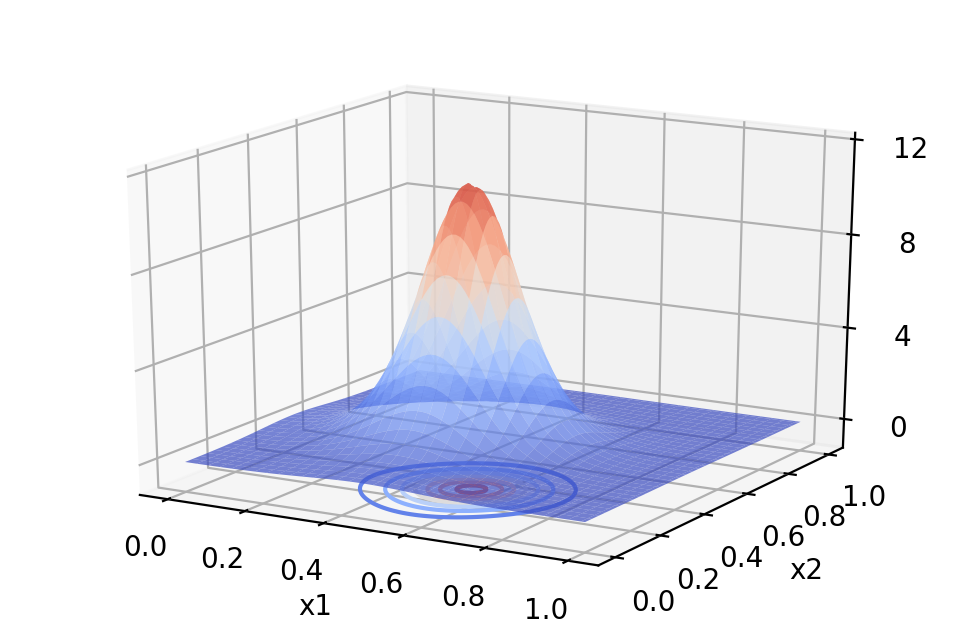} \\
$t=0$ & $t=0.16$ \\
\includegraphics[width=.46\textwidth]{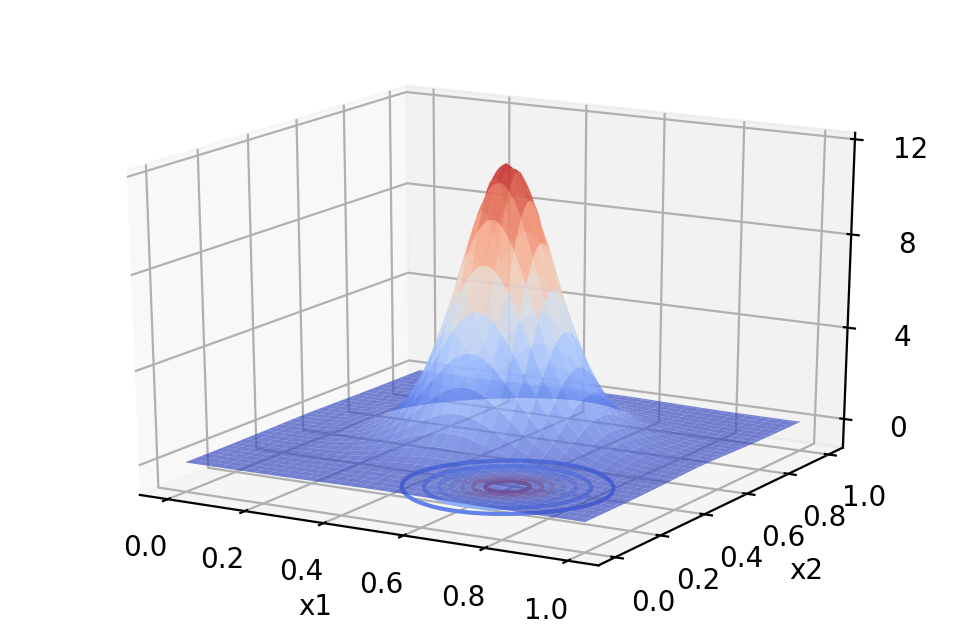} &
\includegraphics[width=.46\textwidth]{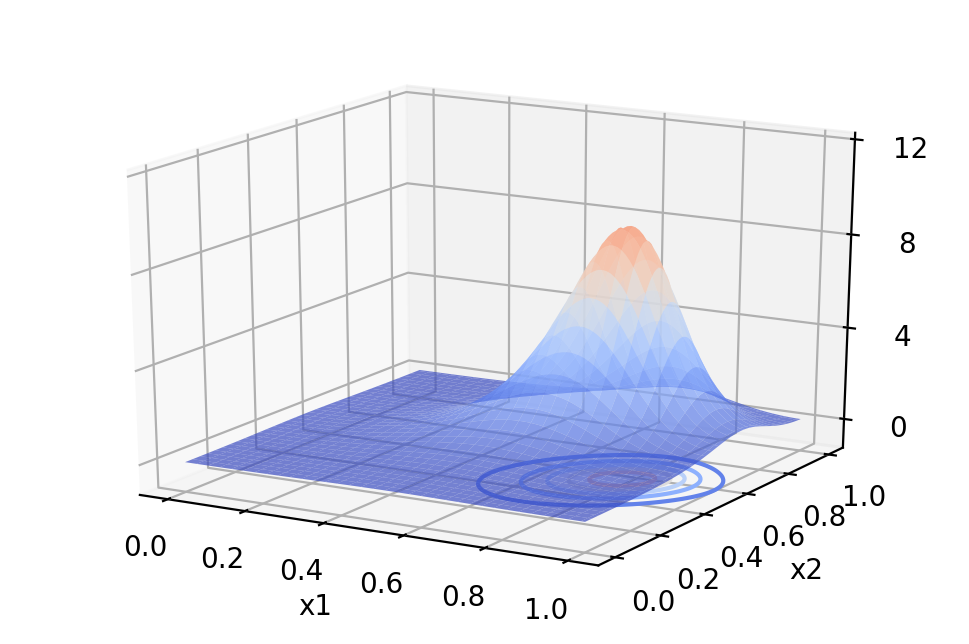} \\
$t=0.33$ & $t=0.5$
 \end{tabular}
\end{figure}
\appendix
\counterwithin{theorem}{section}
\section{Some classical parabolic estimates results}
Consider the linear parabolic equation:
\begin{equation}\label{Linear Parabolic}
		\left\{\begin{split}
		\qquad &	-\partial_t u-\sigma \Delta u+b(t,x)\cdot  \nabla u+c(t,x)u=f(t,x) &\text{ in }Q,\\
			&u(T,x)=u_T(x). &\text{ in }\T^d.
		\end{split}\right.
	\end{equation}	
The following two results are very classical for equations on cylinders with boundary conditions (\cite[Theorem 5.1, p. 320 ]{LSU} and \cite[Theorem 9.1 pp. 341-342]{LSU}). A complete proof of them in the flat torus case can be found in \cite[Appendix pp. 17-18]{cirant2020}.
\begin{proposition}\label{linear estim}
	Let $b\in \mathcal C^{\alpha/2,\alpha}(Q;\R^d)$, $c$ and $f$ belong to $\mathcal C^{\alpha/2,\alpha}(Q)$ and $u_T\in C^{2+\alpha}(\T^d)$.  Then the problem \eqref{Linear Parabolic}
admits a unique solution $u\in \mathcal C^{1+\alpha/2,2+\alpha}(Q)$  and it holds
	\begin{equation*}
		\| u\| _{\mathcal C^{1+\alpha/2,2+\alpha}(Q)}\leq C(\| f\| _{\mathcal C^{\alpha/2,\alpha}(Q)}+\| u_T\| _{\mathcal C^{2+\alpha}(\T^d)}),
	\end{equation*}
	where $C$ depends on the $\mathcal C^{\alpha/2,\alpha}$ norms of $b$, $c$ and remains bounded for bounded values of $T$.\end{proposition}	
		\begin{proposition}\label{linear estim Sobolev}
	Let $r>d+2$, $b\in L^\infty(Q;\R^d)$, $c\in L^\infty(Q)$, $f\in L^r(Q)$ and $u_T\in W^{2-\frac{2}{r}}_r(\T^d)$.  Then the problem \eqref{Linear Parabolic} admits a unique solution $u\in W^{1,2}_r(Q)$ and it holds
	\begin{equation*}
		\| u\| _{W^{1,2}_r(Q)}\leq C(\| f\| _{L^r(Q)}+\| u_T\| _{W^{2-\frac{2}{r}}_r(\T^d)}),
	\end{equation*}
	where $C$ depends on the norms of $b$, $c$ and remains bounded for bounded values of $T$. Moreover, the following embedding \cite[Corollary pp. 342-343]{LSU} holds:
	\begin{equation}\label{embedding}
\|u\|_{\mathcal C^{1-\frac{d+2}{2r},2-\frac{d+2}{r}}(Q)} \leq C\|u\|_{W^{1,2}_r(Q)}.
\end{equation}
\end{proposition}
	Next we introduce some results for the FPK equation. 
\begin{proposition}\label{m stability} Let $r>d+2$, $\|q^{(\iota)}\|_{L^\infty(Q;\R^d)}\leq R$, consider the parabolic equation in divergence form:
\begin{equation}\label{m1}
		\left\{\begin{split}
\qquad &	\partial_t m^{(\iota)}- \sigma \Delta m^{(\iota)} - {\rm{div}} ( m^{(\iota)}q^{(\iota)}) =0&{\rm in}\,\,Q,\\
		&	m^{(\iota)}(0,x)=m_0(x)&{\rm in}\,\,\T^d.
		\end{split}\right.
	\end{equation}
 Denote by $\delta m=m^{(\iota_1)}-m^{(\iota_2)}$ and $\delta q=q^{(\iota_1)}-q^{(\iota_2)}$. Then \par
(i) Let $m_0\in W^1_r(\T^d)$, then there exists a unique solution $m^{(\iota)}$ in $\mathcal H^1_r(Q)$ to \eqref{m1}. Moreover
there exists a constant $C$ depending only on $r,T,d, R$ and $\|m_0\|_{W^1_r(\T^d)}$, such that
$$\| m^{(\iota)}\|_{\mathcal H^1_r(Q)}\leq C,\,\,\|\delta m\|_{\mathcal H^1_r(Q)}\leq C\|\delta q\|_{L^\infty(Q;\R^d)}.$$
\par
(ii) Let $m_0\in W^{2-\frac{2}{r}}_r(\T^d)$ and $\|{\rm{div} }q^{(\iota)}\|_{L^\infty(Q)}\leq R_1$. Then there exists a unique solution $m^{(\iota)}$ in $W^{1,2}_r(Q)$ to \eqref{m1}, moreover for a constant $C$ depending only on $r,T,d, R, R_1$ and $\|m_0\|_{W^{2-\frac{2}{r}}_r(\T^d)}$,
	$$\|\delta m\|_{W^{1,2}_r(Q)}\leq C\big(\|\delta q\|_{L^\infty(Q;\R^d)}+\|{\rm{div} }\delta q^{(\iota)}\|_{L^\infty(Q)}\big).$$	\par
(iii) Let $m_0$ satisfies ($M$), $q^{(\iota)}\in \mathcal C^{\alpha/2,\alpha}(Q;\R^d)$ and $\|{\rm{div} }q^{(\iota)}\|_{L^\infty(Q)}\leq R_1$. Then $m(t,x)>0$ for all $(t,x)\in Q$.	
	 \end{proposition}
\begin{proof}
$\| m^{(\iota)}\|_{\mathcal H^1_r(Q)}\leq C$ can be obtained from \cite[Proposition 2.6]{Cirant2019}. It is clear  
\begin{equation}\label{deltam}
\partial_t\delta m-\sigma \Delta \delta m-\text{div}(q^{(\iota_1)}\cdot \delta m)=\text{div}(\delta q\cdot m^{(\iota_2)}),
\end{equation}
by \cite[Theorem 3.1]{ct}, we have $\|\delta m\|_{L^\infty(0,T;L^r(\T^d))}\leq C\|\delta q\|_{L^\infty(Q;\R^d)}$, where $C$ depends only on  $r,T,d, R$ and $\|m^{(\iota_2)}\|_{L^r(Q)}$. We can use the same duality arguments as in \cite[Proposition 2.4, pp. 7-8]{Cirant2019} to obtain similar estimates on $\|\nabla \delta m\|_{L^r(Q;\R^d)}$ and $\|\partial_t\delta m\|_{(W^{0,1}_{r'}(Q))'}$, with constants depending on $r,T,d, R$ and $\|m_0\|_{W^1_r(\T^d)}$.   Then the bound on $\|\delta m\|_{\mathcal H^1_r(Q)}$ follows.
 (ii) can be obtained analogously by considering \eqref{deltam} and using \cref{linear estim Sobolev}. By standard results, $m(t,x)\geq 0$ for all $(t,x)\in Q$. Hence, let $\varrho(0,x)=m_0(x)$, $\varrho$ be the unique classical solution to 
$$
\partial_t \varrho- \sigma \Delta \varrho - q^{(\iota)}\nabla \varrho+ R_1\varrho=0\,\,\,{\rm in}\,\,Q,
$$
from comparison principle for strong solutions in $W^{1,2}_r(Q)$ (Theorem 7.1 of \cite{LSU}, p. 181), we have $m(t,x)\geq \varrho(t,x)$ for all $(t,x)\in Q$. From strong parabolic maximum principle (e.g. \cite[Theorem 12, p. 399]{evans2022partial}),  $\varrho(t,x)>0$ for all $(t,x)\in Q$.
\end{proof}
\begin{remark}\label{H embedding}
From \cite[Theorem A.3 (iii)]{Meta} and \cite[Proposition 2.1 (iii)]{Cirant2019}, for $r>d+2$ the space $\mathcal H^1_r(Q)$ is continuously embedded in $\mathcal C^{\alpha/2,\alpha}(Q)$, for some $\alpha\in (0,1)$. Hence convergence in $\mathcal H^1_r(Q)$ implies uniform convergence. 
\end{remark}

\textbf{Acknowledgement} \textit{We want to thank Fabio Camilli for reading the manuscripts and giving many helpful comments}.

\end{document}